\documentclass{siamltex}

\usepackage{amssymb,amsmath,amsfonts,graphicx}

\usepackage{float}
\usepackage[caption = false]{subfig}
\usepackage{hyperref}

\usepackage{pgfplots}
\pgfplotsset{compat=1.13}

\usepackage{booktabs} 
\usepackage{graphicx}
\usepackage{algpseudocode}
\usepackage[ruled]{algorithm}

\newcounter{algo}[section]

\newcommand{\IR}{{\mathbb{R}}}

\newtheorem{remark}[theorem]{Remark}

\begin{document}       

\title{A relaxed interior point method for low-rank semidefinite programming problems
with applications to Matrix Completion
\thanks{The work of the first and the third 
author was supported by {\em Gruppo Nazionale per il Calcolo Scientifico} 
(GNCS-INdAM) of Italy. 
The work of the second author was supported
by EPSRC Research Grant EP/N019652/1.}}{}


\author{Stefania Bellavia\thanks{Dipartimento di Ingegneria Industriale, 
        Universit\`a degli Studi di Firenze, viale Morgagni 40, 
        50134 Firenze, Italia. Member of the INdAM Research Group GNCS. E-mail: stefania.bellavia@unifi.it} \and
Jacek Gondzio\thanks{School of Mathematics, The University of Edinburgh, 
        James Clerk Maxwell Building, The King’s Buildings, 
        Peter Guthrie Tait Road, Edinburgh, EH9 3FD, UK. E-mail: j.gondzio@ed.ac.uk}
\and Margherita Porcelli\thanks{Dipartimento di Matematica, 
        Universit\`a di Bologna, Piazza di Porta San Donato 5, 
        40126 Bologna, Italia. Member of the INdAM Research Group GNCS. E-mail: margherita.porcelli@unibo.it}
        }
\date{\today}

\today
\maketitle
\begin{abstract} 
A new relaxed variant of interior point method for low-rank semidefinite 
programming problems is proposed in this paper. The method is a step 
outside of the usual interior point framework. In anticipation 
to converging to a low-rank primal solution, a special nearly low-rank 
form of all primal iterates is imposed. To accommodate such a (restrictive) 
structure, the first order optimality conditions have to be relaxed 
and are therefore approximated by solving an auxiliary least-squares 
problem. The relaxed interior point framework opens numerous possibilities 
how primal and dual approximated Newton directions can be computed. 
In particular, it admits the application of both the first- and the 
second-order methods in this context. The convergence of the method 
is established. 
A prototype implementation is discussed and encouraging preliminary 
computational results are reported for solving the SDP-reformulation 
of matrix-completion problems.
\end{abstract}



\section{Introduction}
\label{Intro}
%
We are concerned with an application of an interior point method (IPM) 
for solving large, sparse and specially structured positive semidefinite 
programming problems (SDPs). 

Let $S\IR^{n\times n}$ denote the set of real symmetric matrices 
of order $n$ and let $U \bullet V$ denote the inner product between 
two matrices, defined by $trace (U^T V)$.
Consider the standard semidefinite programming (SDP) problem 
in its primal form
\begin{equation}\label{sdp_primal}
  \begin{array}{ll}
     \min & C \bullet  X \\
     \mbox{s.t. } & A_i \bullet X=b_i\;\;\; i=1,\ldots,m \\
                  & X \succeq 0, \\
  \end{array}
\end{equation}
where $A_i, C \in S\IR^{n\times n}$ and $b\in \IR^m$ are given and 
$X \in S\IR^{n\times n}$ is unknown and assume that matrices $A_i, i=1,2,\dots,m$ 
are linearly independent, that is $\sum_{i=1}^m d_i A_i = 0$ implies $d_i=0$, $i=1,\ldots,m$.
The dual form of the SDP problem associated with (\ref{sdp_primal}) is:
\begin{equation}\label{sdp_dual}
  \begin{array}{ll}
  \max & b^T y \\
     \mbox{s.t. } & \sum_{i=1}^m y_i A_{i} + S = C \\ 
     & S \succeq 0,  
  \end{array}
\end{equation}
where $y \in \IR^m$ and $S \in S\IR^{n\times n}$. 

The number of applications which involve semidefinite programming problems 
as a modelling tool is already impressive \cite{ToddSurvey,Boydsurvey96} 
and is still growing. Applications include problems arising in engineering, 
finance, optimal control, power flow, various SDP relaxations 
of combinatorial optimization problems, matrix completion or other applications 
originating from modern computational statistics and machine learning. 
Although the progress in the solution algorithms for SDP over the last 
two decades was certainly impressive (see the books on the subject 
\cite{AnjosLassere12,deKlerk}), the efficient solution of general 
semidefinite programming problems still remains a computational challenge. 

Among various algorithms for solving (linear) SDPs, interior point methods 
stand out as reliable algorithms which enjoy 
enviable convergence properties
%
%
and usually provide accurate solutions within reasonable time. 
However, when sizes of SDP instances grow, traditional IPMs which require 
computing exact Newton search directions hit their limits. Indeed, the effort 
required by the linear algebra in (standard) IPMs may grow as fast 
as ${\cal O}(n^6)$.

Although there exists a number of alternative approaches to interior 
point methods, such as for example 
\cite{BurerMonteiro2003,BurerMonteiro2005,KocvaraStingl}, 
which can solve certain SDPs very efficiently, they usually come 
with noticeably weaker convergence guarantees.
%
%
Therefore there is a need to develop faster IPM-based techniques 
which could preserve some of the excellent theoretical properties 
of these methods, but compromise on the other features in quest 
for practical computational efficiency. Customized IPM methods 
have been proposed for special classes of problems. 
They take advantage of sparsity and structure of the problems, 
see e.g. \cite{BellaviaGondzioPorcelli18,benson2000,GilbertHansson03,KKB07,LiuVanden2009,TohKojima2002}  
and the references in \cite{Andersen11}.

In  this paper we focus on problems in which 
the primal variable $X$ is expected to be {\it low-rank} at optimality.
Such situations are common in relaxations of combinatorial optimization 
problems \cite{benson2000}, for example in maximum cut problems \cite{GoemansWilliamson}, 
as well as in matrix completion problems \cite{CaiCandesShen2010}, 
general trust region problems and quadratically 
constrained quadratic problems in complex variables \cite{lemon2016}.
We exploit the structure of the sought solution and relax the rigid structure of IPMs 
for SDP. In particular we propose to weaken the usual connection 
between the primal and dual problem formulation and exploit any special 
features of the primal variable $X$.
However, the extra flexibility added to the interior point method 
comes at a price: the worst-case polynomial complexity has to be 
sacrificed in this case.

Rank plays an important role in semidefinite programming. 
For example, every polynomial optimization problem has a natural SDP 
relaxation, and this relaxation is exact when it possesses a rank-1 
solution \cite{lemon2016}. On the other hand, for any general problem 
of the form (\ref{sdp_primal}), there exists an equivalent formulation 
where an additional bound $r$ on the rank of $X$ may be imposed 
as long as $r$ is not too small \cite{BurerMonteiro2005}. 
More specifically, under suitable assumptions, there exists an optimal 
solution $X^*$ of (\ref{sdp_primal}) with rank $r$ satisfying $r(r+1)/2 \le m$. 
There have been successful attempts to identify low rank submatrices
in the SDP matrices and eliminate them with the aim to reduce the rank
and hence the difficulty of solving an SDP. A technique called
{\it facial reduction} \cite{HuangWolkowiczJoGO2018} has been analysed
and demonstrated to work well in practice.
Interestingly, when positive semidefinite programs are solved using 
interior-point algorithms, then because of the nature of logarithmic 
barrier function promoting the presence of nonzero eigenvalues, 
the primal variable $X$ typically converges to a maximum-rank 
solution \cite{guler1993,lemon2016}. However, in this paper we aim 
at achieving the opposite. We want to design an interior point method  
which drives the generated sequence of iterates to converge to a low-rank 
solution. We assume that constraint matrices are sparse 
and we search 
for a solution $X$ of rank $r$ of the form $X= U U^T$ 
with $U\in \IR^{n\times r}$. \\

Special low-rank structure of $X$ may be imposed directly in problem 
(\ref{sdp_primal}), but this excludes the use of an interior point algorithm 
(which requires all iterates $X$ to be strictly positive definite). 
Burer and Monteiro \cite{BurerMonteiro2003,BurerMonteiro2005} and 
their followers \cite{Boumal2016, Boumal2018} have used such an approach with great 
success. Namely, they have substituted $U U^T$ for $X$ in (\ref{sdp_primal}) 
and therefore have replaced it with the following nonlinear programming 
problem
\begin{equation}
  \label{sdp_monteiro}
  \begin{array}{ll}
     \min & C \bullet  (U U^T) \\
     \mbox{s.t. } & A_i\bullet (U U^T) = b_i \;\;\; i=1,\ldots,m,
  \end{array}
\end{equation}
with $U \in \IR^{n \times r}$. Although such transformation removes 
the difficult positive definiteness constraint (it is implicit as $X = U U^T$), 
the difficulty is shifted elsewhere as both the objective and constraints 
in (\ref{sdp_monteiro}) are no longer linear, but instead quadratic 
and in general non-convex. In comparison with a standard IPM the method 
proposed in \cite{BurerMonteiro2003,BurerMonteiro2005} and applied 
to solve large-scale problems enjoys substantially reduced memory 
requirements and very good efficiency and accuracy.
However, due to nonconvexity of (\ref{sdp_monteiro}), local methods 
may not always recover the global optimum. In \cite{Boumal2016,Boumal2018}
authors showed that, despite the non-convexity, first- and second-order
necessary optimality conditions are also sufficient, provided that 
rank $r$ is large enough and constraints satisfy some regularity conditions.
That is, when applied to several classes of SDPs, the low-rank 
Burer-Monteiro formulation is very unlikely to converge to any 
spurious local optima.

In this paper we propose a different approach. 
We would like to preserve as many of the advantageous properties 
of interior point methods as possible and expect to achieve it 
by (i) working with the original problem (\ref{sdp_primal}) and 
(ii) exploiting the low-rank structure of $X$. Knowing that 
at optimality $X$ is low-rank we impose a special form 
of the primal variable throughout the interior point algorithm 
$$
  X = \mu I_n + U U^T,
$$
with $U\in \IR^{n\times r}$, for a given $r>0$ and $\mu$ denoting  
the barrier term. 
Hence $X$ is full rank (as required by IPM), but approaches the low-rank 
matrix as $\mu$ goes to zero. Imposing such special structure of $X$ 
offers an advantage to an interior point algorithm: it can work 
with an object of size $n r$ rather than a full rank $X$ of size $n^2$. 
We have additionally considered an adaptive choice of $r$ assuming 
that this rank may not be known a priori. Indeed, the method can start 
with $r$ equal to 1 or 2 and gradually increase $r$ to the necessary 
minimum rank ({\em target rank}). 
Remarkably, the method can also handle problems with nearly-low-rank 
solution, as the primal variable is not assumed to be low-rank along 
the iterations, but it is gradually pushed to a low-rank matrix. 
Finally, the presence of the {\it perturbation term} $\mu I$ allows 
to deal with possibly noisy right-hand side $b$ as well. 
We also further relax the rigid IPM structure.
Starting from a dual feasible approximation, we
dispose of dual slack variable $S$
and avoid computations 
which would involve large Kronecker product matrices of dimension 
$n^2 \times n^2$ (and that in the worst case might require up to ${\cal O}(n^6)$ 
arithmetic operations).  We investigate the use of both first- and 
second-order methods for the step computation and devise matrix-free 
implementations of the linear algebra phase arising in the second-order method.
Such implementations are well-suited to the solution 
of SDP relaxations of matrix completion problems \cite{CandesRecht2009}.

The paper is organised as follows. 
After a brief summary of notation used in the paper provided 
at the end of this section, in Section \ref{sectDBA} we present 
the general framework and deliver some theoretical insights 
into the proposed method. 
In Section \ref{RankUpDown} we explain the mechanism which allows 
to adaptively reveal the rank of the minimum rank solution matrix $X$. 
The proposed approach offers significant flexibility in the way 
how Newton-like search directions are computed. They originate 
from a solution of a least squares problem. 
We see  it in detail in Section \ref{SolveNLS}. 
Next, in Section \ref{sec:mc} we discuss the properties of low-rank 
SDPs arising in matrix completion problems and in Section \ref{sec:exp} 
we present preliminary computational results obtained with a prototype 
Matlab implementation of the new algorithm. {\color{black}
We also provide a comparison of its efficiency against {\sc OptSpace} 
\cite{optspace1,optspace2} when both methods are applied to various 
instances of matrix completion problems.}
Finally, in Section \ref{Concls} we give our conclusions.
Appendix \ref{sec:app} contains some notes on the Kronecker 
product of two matrices and on matrix calculus. \\

\noindent
{\bf Notation.} 
The norm of the matrix associated with the inner product between 
two matrices $U \bullet V = trace (U^T V)$ is the Frobenius norm, 
written $\|U\|_F := (U\bullet U )^{1/2}$, while $\|U\|_2$ denotes 
the L$_2$-operator norm of a matrix. Norms of vectors will always 
be Euclidean. The symbol $I_p$ denotes the identity matrix of dimension $p\times p$.

Let ${\cal A}$ be the linear operator ${\cal A}: S\IR^n\rightarrow \IR^m$ defined by
$$
{\cal A}(X)=(A_i\bullet X)_{i=1}^m\in \IR^m,
$$
with $A_i \in S\IR^{n\times n}$, then its transposition ${\cal A}^T$ 
$$
{\cal A}^Tv=\sum_{i=1}^m v_i A_i.
$$
Moreover, let $A^T$ denote  the matrix representation of ${\cal A}^T$ 
with respect to the standard bases of $\IR^n$, that is
\begin{equation}\label{AT}
A^T := [vec(A_1), vec(A_2), \dots ,vec(A_m)] \in \IR^{n^2 \times m },
\end{equation}
and
$$
{\cal A}(X) = A\, vec(X) \quad \mbox{ and } \quad {\cal A}^Tv=mat(A^{T} v),
$$
where $mat$ is the ``inverse'' operator to $vec$ 
(i.e., $mat(vec(A_i)) = A_i \in S\IR^{n\times n}$) and 
the $vec$ operator is such that 
$vec(A)$ is the vector of columns of $A$ stacked one under the other.

\section{Relaxed interior point method for low-rank SDP}
\label{sectDBA}
%
Interior point methods for semidefinite programming problems work 
with the perturbed first-order optimality conditions 
for problems (\ref{sdp_primal})-(\ref{sdp_dual}) given by:
\begin{equation}
\label{F_mu}
   F_\mu(X,y,S) = 
   \left (
   \begin{array}{c}
      {\cal A}^T y + S - C \\
      {\cal A} (X) - b \\
      XS-\mu I_n
   \end{array}
   \right )
   =0, \ \mu > 0, \
   S \succeq 0 \ 
   X \succeq 0.
\end{equation} 
A general IPM involves a triple $(X,y,S)$, performs steps in Newton 
direction for (\ref{F_mu}), and keeps its subsequent iterates 
in a neighbourhood of the central path \cite{AnjosLassere12,deKlerk}.
The convergence is forced by gradually reducing the barrier term $\mu$.
However, having in mind the idea of converging to a low-rank solution, 
we find such a structure rather restrictive and wish to relax it. 
This is achieved by removing explicit $S$ from the optimality conditions
and imposing a special structure of $X$. 

Substituting $S = C - {\cal A}^Ty $ from the first equation into the third 
one, we get
\begin{equation}
\label{F_mu2}
   \left (
   \begin{array}{c}
      {\cal A} (X) - b \\
      X (C - {\cal A}^T y) -\mu I_n 
   \end{array}
   \right )
   =0, \ \mu > 0, \
C - {\cal A}^Ty \succeq 0,\   X \succeq 0.
\end{equation}
Next, following the expectation that at optimality $X$ has rank $r$, 
we impose on $X$ the following special structure 
\begin{equation}
   \label{Xlow}
   X = \mu I_n + U U^T,
\end{equation}
with $U \in \IR^{n \times r}$, for a given $r>0$.
We do not have any guarantee that there exists a solution 
of \eqref{F_mu2} with such a structure, but we can consider 
the least-square problem:
\begin{equation}
   \label{least-square}
      \min_{U,y} \phi_\mu(U,y) := \frac{1}{2}\|F^r_{\mu}(U,y)\|^2, 
\end{equation}
where $F^r_{\mu}(U,y): \IR^{n\times r } \times \IR^{m}\rightarrow \IR^{n^2+m}$ 
is given by
\begin{equation} 
   \label{F_mu3}
   F^r_{\mu}(U,y)=
   \left (
   \begin{array}{c}
      {\cal A} (\mu I_n + U U^T) - b \\
      vec((\mu I_n + U U^T) (C - {\cal A}^T y) - \mu I_n )
   \end{array}
   \right ), \ \mu > 0.
\end{equation}
The nonlinear function $F^r_{\mu}(U,y)$ has been obtained substituting 
$X = \mu I_n + U U^T$ in \eqref{F_mu2} after vectorization of the second block.
The associated system $F^r_{\mu}(U,y)=0$ is overdetermined with $(m+n^2)$ equations 
and ($nr+m$) unknowns $(U,y)$. In the following, for the sake of simplicity, 
we identify $\IR^{n\times r} \times \IR^{m}$ with $\IR^{nr+m}$.

It is worth mentioning at this point that the use of least-squares 
type solutions to an overdetermined systems arising in interior point 
methods for SDP was considered in \cite{GN1,GN2}. Its primary objective 
was to avoid symmetrization when computing search directions and 
the least-squares approach was applied to a standard, complete 
set of perturbed optimality conditions (\ref{F_mu}).

We propose to apply to problem (\ref{least-square}) a similar framework 
to that of interior point methods, namely: Fix $\mu$, iterate 
on a tuple $(U,y)$, and make steps towards a solution to (\ref{least-square}). 
This opens numerous possibilities. One could for example 
compute the search directions for both variables at the same time, 
or alternate between the steps in $U$ and in $y$. 

Bearing in mind that (\ref{F_mu}) are the optimality conditions for \eqref{sdp_primal}
and assuming that a rank $r$ optimal solution of \eqref{sdp_primal} exists,
we will derive an upper bound on the optimal residual 
of the least-squares problem (\ref{least-square}).
Assume that a solution $(X^*,y^*,S^*)$ of the KKT conditions exists 
such that $X^*=U^*(U^*)^T$, $U^* \in \IR^{n \times r}$, that is
\begin{equation}\label{optim}
   \begin{array}{l}
      {\cal A} (U^*(U^*)^T) = b \\
      S^* = C - {\cal A}^T y^* \succeq 0 \\
      U^*(U^*)^T S^* = 0. 
   \end{array}
\end{equation}
Then evaluating (\ref{F_mu3}) at $(U^*,y^*)$  and using \eqref{optim} we get
$$
   F^r_{\mu} (U^*,y^*) = 
   \left (
   \begin{array}{c}
      {\cal A} (\mu I_n)+  {\cal A}(U^* (U^*)^T) - b\\
     vec( (\mu I_n + U^* (U^*)^T) (C - {\cal A}^T y^*) -\mu I_n )
   \end{array}
   \right )=
   \left (
   \begin{array}{c}
      \mu{\cal A} ( I_n) \\
      \mu\, vec(S^* - I_n) 
   \end{array}
   \right ). 
$$
Consequently, we obtain the following upper bound 
for the residual of the least-squares problem \eqref{least-square}: 
\begin{eqnarray}
 \phi_{\mu}(U^*,y^*)&=&
     \frac{1}{2} \|{\cal A} (U^* (U^*)^T+\mu I_n) - b\|^2 
    + \frac{1}{2}\|(U^*(U^*)^T+\mu I_n) (C - {\cal A}^T y^*) - \mu I_n\|_F^2 \nonumber \\
      &  =  & \omega^* \mu^2,
       \label{bound_res}
   \end{eqnarray}
   where
   \begin{equation}\label{omega}
    \omega^*=\frac{1}{2}\|{\cal A} (I_n)\|_2^2 + \frac{1}{2}\|S^*-I_n\|^2_F.
\end{equation}
%
%

Assuming to have an estimate of $\omega^*$ 
we are now ready to sketch in Algorithm \ref{GF_algo} the general 
framework of a new {\it relaxed} interior point method.
To start the procedure we need an initial guess $(U_0,y_0)$ 
such that $U_0$ is full column rank and $S_0 = C - {\cal A}^T y_0$ 
is positive definite,  and an initial barrier parameter $\mu_0 > 0$. 
At a generic iteration $k$, given the current barrier parameter 
$\mu_k>0$, we compute an approximate solution $(U_k,\bar y_k)$ 
of \eqref{least-square} such that 
$\phi_{\mu_k}(U_k,\bar y_k)$ is below $\mu_k^2 \omega^*$. Then, the dual 
variable $y_k$ and the dual slack variable $S_k$ are updated as follows:
$$
 \begin{array}{l}
 y_k = y_{k-1} + \alpha_k (\bar y_k - y_{k-1}) \\
 S_k = C - {\cal A}^T y_k  = S_{k-1}-\alpha_k {\cal A}^T (\bar y_k - y_{k-1}) \\
 \end{array}
$$
with $\alpha_k \in (0,1]$ such that $S_k$ remains positive definite. 
We draw the reader's attention to the fact that although the dual variable $S$ 
does not explicitly appear in optimality conditions (\ref{F_mu2}) or (\ref{F_mu3}), 
we do maintain it as the algorithm progresses and make sure that $(S_k,y_k)$ 
remains dual feasible. 
Finally, to complete the major step of the algorithm, the barrier parameter 
is reduced and a new iteration is performed. 

Note that so far we have assumed that there exists a solution 
to \eqref{sdp_primal}  of rank $r$. In case such a solution does not 
exist the optimal residual of the least-squares problem is not guaranteed 
to decrease as fast as $\mu_k^2$. This apparently adverse case can 
be exploited to design an adaptive procedure that increases/decreases $r$ 
without requiring the knowledge of the solution's rank. 
This approach will be described in Section \ref{RankUpDown}.

\begin{algorithm}
\caption{General framework of the Relaxed Interior Point algorithm 
         for solving low-rank SDP}\label{GF_algo}
\begin{algorithmic}[1]
\Require Initial $(U_0,y_0)$ with $U_0\in \IR^{n\times r}$ 
         and $S_0 = C - {\cal A}^T y_0$ positive definite, $\mu_0 >0$, 
 $\gamma\ge \sqrt{\omega^*}$, $\sigma \in (0,1)$.
 \For{ $\;k=1,2,\dots$}
 \State Find $(U_k, \bar y_k)$ such that
$$
\|{\cal A} (U_k U_k ^T+\mu_k I) - b\|^2 
               +  \|(U_kU_k^T+\mu_k I) (C - {\cal A}^T \bar y_k) - \mu_k I_n\|_F^2 
\le \gamma^2 \mu_k^2
$$
\ \ \ \ by approximately solving 
\begin{equation}\label{minphi}
 \min_{(U,y)} \phi_{\mu_k}(U,y).
\end{equation}

\State If $C - {\cal A}^T \bar y_k $ is positive definite set $\alpha_k=1$ otherwise, 
set $\Delta y=\bar y_k- y_{k-1}$, 
\hspace*{15pt}$\Delta S=-{\cal A}^T \Delta y$ and backtrack along $\Delta S$ to ensure
$S_k= S_{k-1}+\alpha_k \Delta S$ positive 
\hspace*{13pt} definite.
\State  Set $y_k= y_{k-1}+\alpha_k \Delta y$ (and $X_k=(U_k U_k ^T+\mu_k I)$). 
\State Set $\mu_k =\sigma \mu_{k-1}$

\EndFor
\end{algorithmic}
\end{algorithm}

In the remaining part of this section we state some of the properties 
of the Algorithm which are essential to make it work in practice. 

First we note that dual constraint is always satisfied by construction 
and the backtracking process at Line 3 is well-defined. This is proved 
 in  Lemma 4 of \cite{BellaviaGondzioPorcelli18} which is repeated below 
for sake of reader's convenience.  

\begin{lemma} \label{lemma_alpha}
Let $\Delta S$ be computed in Step 3 of Algorithm \ref{GF_algo} at iteration $k$ 
and $S_{k-1}$ be computed at the previous iteration $k-1$. Then, there 
exists $\alpha_k\in (0, 1]$ such that $S_k=S_{k-1}+\alpha_k \Delta S$ 
is positive definite.
\end{lemma}
\begin{proof}
Assume that $C - {\cal A}^T \bar y_k $
is not positive definite, otherwise $\alpha_k=1$. 
Noting that 
 $S_{k-1} \succ 0$ by construction, it follows that  $\Delta S$ is indefinite and  
$S_{k-1} + \alpha_k \Delta S \succ 0$ whenever $\alpha_k$ is sufficiently small.
In particular, since $S_{k-1} + \alpha_k \Delta S  = S_{k-1}^{1/2} (I_n + \alpha_k S_{k-1}^{-1/2} \Delta S S_{k-1}^{-1/2}) S_{k-1}^{1/2}$,
the desired result holds with 
$$
\alpha_k<  \frac{-1}{\lambda_{\min}(S_{k-1}^{-1/2} \Delta S S_{k-1}^{-1/2})}.
$$
\end{proof}

Note that if backtracking is needed (i.e. $\alpha_k<1$) to maintain 
the positive definiteness of the dual variable, then after updating 
$S_k$ in Step 5 the centrality  measure  $\|X_k S_k - \mu_k I_n\|$ may increase and it is not guaranteed 
to remain below $\gamma \mu_k$. Indeed, by setting
$S_k = \bar S_k - (1-\alpha_k) \Delta S$
with $\bar S_k = C - {\cal A}^T \bar y_k$, we have:
\begin{equation}\label{centr_inc}
 \|X_k S_k - \mu_k I_n\|^2_F \le\gamma^2 \mu^2_k + (1-\alpha_k)^2 \|X_k\Delta S\|^2_F 
     - 2 (1-\alpha_k) (X_k \bar S_k - \mu_k I_n)\bullet (X_k \Delta S),
\end{equation}
that is the centrality measure  may actually increase along the iterations 
whenever $\alpha_k$ does not approach one as $\mu_k$ goes to zero.
In the following we analyse the convergence properties of Algorithm \ref{GF_algo} 
when this adverse situation does not occur, namely under the following assumption:
%
%
\vskip 5 pt
\noindent
{\em Assumption 1}. Assume that there exists $\bar k>0$ 
such that $\alpha_k= 1$ for $k\ge \bar k$. 
\vskip 5 pt 
To the best of authors knowledge, it does not seem possible 
to demonstrate that eventually $\alpha_k$ is equal to one. 
This is because we impose a special form of $X$ in (\ref{Xlow}) 
and make only a weak requirement regarding the proximity 
of the iterate to the central path:  
\begin{equation}\label{neib}
  \|(U_kU_k^T+\mu_k I) (C - {\cal A}^T \bar y_k) - \mu_k I_n\|_F \le \gamma \mu_k
\end{equation}
with $\gamma$ possibly greater than one. 

\begin{proposition}
Let Assumption 1 hold. Assume that a solution of rank $r$ of problem 
(\ref{sdp_primal}) exists and that the sequence $\{U_k,y_k\}$ admits a limit point
$(U^\dagger, y^\dagger)$. Then,   
\begin{itemize}
\item $X^\dagger=U^\dagger (U^\dagger)^T$ is primal feasible, 
\item $X^\dagger S^\dagger=0$ with 
      $ S^\dagger= C- {\cal A}^T y^\dagger $,
\item $S^\dagger$ is positive semidefinite. 
\end{itemize}
\end{proposition}
\vskip 4pt
\begin{proof}
Assume for the sake of simplicity that the whole sequence is converging 
to $(U^\dagger, y^\dagger)$. 
Taking into account that $\lim_{k\rightarrow \infty} \mu_k=0$, it follows
$\lim_{k\rightarrow \infty} U_k (U_k)^T + \mu_k I = (U^\dagger)(U^\dagger)^T$. 
Then $X^\dagger=  (U^\dagger)(U^\dagger)^T$ has at most rank $r$ and it is 
feasible as 
$$
\lim_{k\rightarrow \infty}\|{\cal A} (U_k U_k ^T+\mu_k I) - b\| 
    \le \lim_{k\rightarrow \infty}\gamma \mu_k=0.
$$
Moreover, 
 from \eqref{centr_inc} and Assumption 1 it follows 
$$
\lim_{k\rightarrow \infty}\|(U_kU_k^T+\mu_k I) (C - {\cal A}^T y_k) -\mu_k I_n\|_F = 0,
$$
which implies $X^\dagger S^\dagger=0$ and by construction ensures 
that $S^\dagger$ is positive semidefinite being a limit point 
of a sequence of positive definite matrices. 
\end{proof}

From the previous proposition it follows that $(X^\dagger, y^{\dagger},S^\dagger)$ 
solves \eqref{F_mu}. Moreover, $X^\dagger$ has rank $r$, unless $U^\dagger$ 
is not full column rank. This situation can happen only in the case \eqref{sdp_primal} 
admits a solution of rank smaller than $r$. In what follows for sake of simplicity 
we assume that the limit point $U^\dagger$ is full column rank.

\begin{remark}\label{rem:nearly}
It is worth observing that due to the imposed structure 
of matrices (\ref{Xlow}) all iterates $X_k$ are full rank, 
but asymptotically they approach rank $r$  matrix. Moreover, 
the minimum distance of $X_k$ to a rank $r$ matrix is given by $\mu_k$, i.e.,
\begin{equation}\label{dist:rank}
   \min_{rank(Y)=r} \|X_k-Y\|_2=\mu_k,
\end{equation}
and the primal infeasibility 
is bounded by $\gamma \mu_k$. This allows us to use the proposed 
methodology also when the sought solution is close to a rank $r$ matrix 
(``nearly low-rank'') and/or some entries in vector $b$ are corrupted 
with a small amount of noise.

\end{remark}

\section{Rank updating/downdating}
\label{RankUpDown}
%
The analysis carried out in the previous section requires the knowledge 
of $\gamma\ge \sqrt{\omega^*}$ and of the rank $r$ of the sought solution. 
As the scalar $\gamma$ is generally not known, at a generic iteration $k$ 
the optimization method used to compute an approximate minimizer 
of \eqref{least-square} is stopped when a chosen first-order criticality 
measure $\psi_{\mu}$ goes below the threshold $\eta_2 \mu_k$ where 
$\eta_2$ is a strictly positive constant.  
This way, the accuracy in the solution of \eqref{least-square} 
increases as $\mu_k$ decreases.
For $\psi_{\mu}$, we have chosen $\psi_{\mu}(U,y)=\|\nabla \phi_{\mu}(U,y)\|_2$.

Regarding the choice of the rank $r$, there are situations where 
the rank of the sought solution is not known. Below we describe 
a modification of Algorithm \ref{GF_algo} where, starting from 
a small rank $r$, the procedure adaptively increases/decreases it. 
This modification is based on the observation that if a solution 
of rank $r$ exists the iterative procedure used in Step 2, should 
provide a sequence $\{U_k\}$ such that the primal infeasibility 
also decreases with $\mu_k$. Then, at each iteration the ratio 
\begin{equation}\label{rho}
  \rho_k = \frac{\|{\cal A} (U_k U_k ^T+\mu_k I) - b\|_2}
                {\|{\cal A} (U_{k-1} U_{k-1} ^T+\mu_{k-1} I) - b\|_2}
\end{equation}
is checked. 
If this ratio is larger than $\eta_1$, where $\eta_1$ is a given 
constant in $(\sigma,1)$ and $\sigma$ is the constant used 
to reduce $\mu_k$, then the rank $r$ is increased  by some fixed $\delta_r>0$ as the procedure 
has not been able to provide the expected  decrease in the primal 
infeasibility.  
After an update of rank, the parameter $\mu_k$ is not changed and 
$\delta_r$ extra columns are appended to the current $U_k$. 
As a safeguard, also a downdating strategy can be implemented. 
In fact, if after an increase of rank, we still have $\rho_k>\eta_1$ 
then we come back to the previous rank and inhibit rank updates in all 
subsequent iterations. 

This is detailed in Algorithm \ref{GF_algo_upd} where we borrowed 
the Matlab notation. Variable {\tt update\_r} is an indicator 
specifying if at the previous iteration the rank was increased 
({\tt update\_r = up}), decreased ({\tt update\_r = down}) 
or left unchanged ({\tt update\_r = unch}).

\begin{algorithm}
\caption{Relaxed Interior Point Algorithm 
         for Low Rank SDP (IPLR)}\label{GF_algo_upd}
\begin{algorithmic}[1]
\Require The initial rank $r$, the rank increment/decrement $\delta_r$, initial $(y_0,U_0)$ 
         with $U_0 \in \IR^{n \times r}$ and $y_0 \in \IR^m$ such that 
         $S_0 = C - {\cal A}^T y_0$ is positive definite, $\mu_0 >0$, $\sigma \in (0,1)$,
         $\eta_1\in (\sigma,1)$, $\eta_2>0$, 
         $\epsilon >0$. 
 \State {\tt update\_r = unch}.
 \For {$\;k=1,2,\dots $}

\State Find an approximate minimizer $(U_k, \bar y_k)$ 
of $\phi_{\mu_k}(U,y)$ such that
$$
\| \nabla \phi_{\mu_k}(U_k,\bar y_k)\| \le \eta_2\, \mu_k
$$
\State If $C - {\cal A}^T y_k $ is positive definite set $\alpha_k=1$ otherwise, 
set $\Delta y=\bar y_k- y_{k-1}$, 
\hspace*{15pt}$\Delta S=-{\cal A}^T \Delta y$ and backtrack along $\Delta S$ 
to ensure $S_k= S_{k-1}+\alpha_k \Delta S$ positive 
\hspace*{13pt} definite.
\State  Set $y_k= y_{k-1}+\alpha_k \Delta y$.
\If{$\mu_k< \epsilon$}
\State \textbf{return} $X_k=U_k U_k ^T+\mu_k I$.
\Else
\State Compute $\rho_k$ given in \eqref{rho}. 

\If{$\rho_k>\eta_1$}
  \If{{\tt update\_r = unch}}
    \State Set $r=r+\delta_r$ and {\tt update\_r = up}
     \State Set $U_k=[U_k,e_{r-\delta_r+1}, \dots, e_{r}]$, 
            where $e_i$ is the $i$-th vector of the \hspace*{55pt} canonical 
            basis for $i = r-\delta_r+1, \dots, r$
  \ElsIf{ {\tt update\_r = up}}
  \State Set $r=r-\delta_r$ and {\tt update\_r = down}
  \State Set $U_k=U_{k-1}[1:r]$ and $y_k=y_{k-1}$ 
  \EndIf
  \Else 
  \State Set {\tt update\_r = unch}  
\EndIf
 \If  {{\tt update\_r = unch}}
 \State Set  $\mu_{k+1}=\sigma  \mu_{k}$
\EndIf

\EndIf

\EndFor
\end{algorithmic}
\end{algorithm}


The initial rank $r$ should be chosen as the rank of the solution 
(if known) or as a small value (say 2 or 3) if it is unknown. 
The dimension of the initial variable $U_0$ is then defined accordingly.
Since, for given $\epsilon$ and $\sigma$, the number of iterations to satisfy $\mu_k < \epsilon$
at Line 6
is predefined, the number of rank updates is predefined as well. 
Therefore, if an estimate of the solution rank is known, 
one should use it in order to define a suitable initial $r$.

\section{Solving the nonlinear least-squares problem}
\label{SolveNLS}

In this section we investigate the numerical solution of the nonlinear 
least-squares problem \eqref{least-square}. 


Following the derivation rules recalled in Appendix \ref{sec:app}, 
we compute the Jacobian matrix $J_{\mu_k} \in \IR^{(n^2+m) \times (nr+m)}$ 
of $ F^r_{\mu_k}$ which takes the following form: 
$$
J_{\mu_k}(U,y) = \left(
      \begin{array}{cc}
      A Q              &    0 \\
      ((C - {\cal A }^Ty) \otimes I_n) Q  &  -(I_n \otimes (\mu_k I_n + U U^T)) A^T
      \end{array} 
      \right),  
$$
where 
\begin{equation}\label{Qk}
Q = (U\otimes I_n) + (I_n \otimes U) \Pi_{nr} \in \IR^{n^2\times nr},
\end{equation}
and $\Pi_{nr} \in \IR^{nr\times nr}$ is the unique permutation matrix 
such that $vec(B^T) = \Pi_{nr} vec(B)$ for any $B \in \IR^{n \times r}$, 
see Appendix \ref{sec:app}.

In order to apply an iterative method for approximately solving \eqref{least-square} 
we need to perform the action of $J_{\mu_k}^T$ on a vector to compute 
the gradient of $\phi_{\mu_k}$. The action of $J_{\mu_k}$ 
on a vector is also required in case one wants to apply 
a Gauss-Newton approach (see Section \ref{sec:gs}).
In the next section we will discuss how these computations are carried out.

\subsection{Matrix-vector products with blocks of $J_{\mu_k}$}
First, let us denote the Jacobian matrix blocks as follows:
\begin{eqnarray}
 J_{11} & = &  A Q = A (( U \otimes I_n) + (I_n \otimes U) \Pi_{nr}) \in\IR^{m \times nr} \\
 J_{21} & = &  ((C - {\cal A }^Ty) \otimes I_n) Q  =  (S \otimes I_n) Q  \in\IR^{n^2 \times nr} \\
 J_{22} & = & - (I_n \otimes (\mu_k I_n + U U^T)) A^T = - (I_n \otimes X) A^T \in\IR^{n^2 \times m} 
\end{eqnarray}
Below we will show that despite $J_{\mu_k}$ blocks contain matrices 
of dimension $n^2\times n^2$, matrix-vector products can be carried 
out without involving such matrices and the sparsity of the constraint 
matrices can be exploited. We will make use of the properties of the 
Kronecker product \eqref{kron1}-\eqref{kron3} and assume that 
if $v \in \IR^{nr}$ and $\tilde z \in \IR^{n^2}$ 
then $ mat(v) \in \IR^{n\times r}$ and $ mat(\tilde z) \in \IR^{n\times n}$.
\begin{itemize}
\item 
Let $v\in \IR^{nr}$ and $w\in \IR^m$ and let us consider the action 
of $J_{11}$ and $J_{11}^T$ on $v$ and $w$, respectively: 
\begin{eqnarray}\label{J11Vk}
   J_{11} v &=& {\cal A} ( mat(v) U^T + U mat(v)^T ) = 
   \left( A_i \bullet V \right)_{i=1}^m
\end{eqnarray}
where 
\begin{equation}\label{Vk}
 V = mat(v) U^T + U mat(v)^T \in \IR^{n\times n}.
\end{equation} 
and 
\begin{eqnarray}
   J_{11}^T w & = & Q^T A^T w= (( U^T \otimes I_n) + \Pi_{nr}^T(I_n \otimes U^T)) A^T w
                    \nonumber \\
              &= &  (( U^T \otimes I_n) + \Pi_{nr}^T(I_n \otimes U^T)) vec ({\cal A}^T w)
                    \nonumber \\
              &= & vec(({\cal A}^T w)U + ({\cal A}^T w)^T U)
                    \nonumber \\
              &= & 2 vec(({\cal A}^T w)U) = 2 vec\left(\sum_{i=1}^m w_i A_i U\right). \label{AU}
\end{eqnarray}

\item 
Let $ v \in \IR^{nr}$ and $\tilde z \in \IR^{n^2}$ and let us consider 
the action of  $J_{21}^TJ_{21}$ and  $J_{21}^T$ on $ v$ and $\tilde z$, 
respectively:
\begin{eqnarray}
   J_{21}^T J_{21} v &=& Q^T (S^2 \otimes I_n) Q v  \nonumber \\
                      &=& vec( (mat( v) U^T + U mat( v)^T ) S^2 U \nonumber \\
                      & & + S^2 (mat( v) U^T + U mat(v)^T)^T U )   
\end{eqnarray}
and
\begin{eqnarray}
   J_{21}^T \tilde z &=& Q^T vec (mat(\tilde z)S) 
                         = vec(mat(\tilde z) S U + S mat(\tilde z)^T U) 
\end{eqnarray}

\item  
Let $w \in \IR^{m}$ and $\tilde z \in \IR^{n^2}$ and let us consider 
the action of $J_{22}$ and $J_{22}^T$ on $w$ and $\tilde z$, respectively:
\begin{eqnarray}
    J_{22} w  &=& -(I \otimes X) A^T w \nonumber \\
                      &=& -vec(X {\cal A }^T w) =- vec(X \sum_{i=1}^m w_i A_i) \label{wA}
\end{eqnarray}
and 
\begin{eqnarray}\label{J22tilZ}
   J_{22}^T \tilde z  &=& -A (I \otimes X) z =- A vec(X mat(\tilde z)) =- {\cal A} (X mat(\tilde z)) 
                   \nonumber \\
               &=&- \left( A_i \bullet \tilde Z \right)_{i=1}^m, 
\end{eqnarray}
with 
\begin{equation}\label{tilZ}
  \tilde Z = (\mu I_n + U U^T)mat(\tilde z).
\end{equation}
   
\end{itemize}

\subsection{Computational effort per iteration}
The previous analysis shows that we can perform all products involving 
Jacobian's blocks handling only $n\times n$ matrices. Moreover, 
if matrices $A_i$ are indeed very sparse their structure can be exploited 
in the matrix-products in \eqref{AU} and \eqref{wA}. 
(Sparsity has been exploited of course in various implementations 
of IPM for SDP, see e.g. \cite{FujisawaEtAl-MP1997}.)
Additionally, only few elements of matrices $V$ in (\ref{Vk}) 
and $\tilde Z$ in (\ref{tilZ}) need to be involved when products 
(\ref{J11Vk}) and (\ref{J22tilZ}) are computed, respectively.
More precisely, denoting with {$nnz(A)$} the number of nonzero entries 
of $A$, we need to compute $nnz(A)$ entries of $V$ and $\tilde Z$ 
defined in \eqref{Vk} and \eqref{tilZ}, respectively. 
Noting that $mat(v) \in \IR^{n\times r}$ and $U^T\in \IR^{r\times n}$, 
the computation of the needed entries of $V$  amounts to $(O(nnz(A)r)$ 
flops. Regarding $\tilde Z$, the computation of the intermediate 
matrix $\hat W= U^T mat(\tilde z) \in \IR^{r \times n}$ costs $O(n^2r)$ 
flops and $nnz(A)$ entries of $U \hat W$ requires $O(nnz(A)r)$ flops.

In Table \ref{table_flop} we provide the estimate flop counts for computing 
various matrix-vector products with the blocks of Jacobian matrix. 
We consider the products that are relevant in the computation of the gradient 
of $\phi_{\mu_k}$ and in handling the linear-algebra phase of the second 
order method which we will introduce in the next section. From the table, 
it is evident that the computation of the gradient of $\phi_{\mu_k}$ 
requires $O(\max\{ nnz(A),n^2\}r+m)$ flops.

\begin{table}[h]
\begin{center}
\begin{small}
\begin{tabular}{l|l}
  Operation & Cost \\
  \hline
  $J_{11} v$          & $O(nnz(A)( r+1))$ \\
  $J_{11}^T w$        & $O(nnz(A)r)$ \\ 
  $J_{21}^T J_{21} v$ & $O(n^2 r)$ \\
  $J_{21}^T \tilde z$ & $O(n^2 r)$ \\
  $J_{22} w$          & $O(n(nnz(A)))$ \\
  $J_{22}^T \tilde z$ & $O(n^2+nnz(A))r$ \\ 
  \hline
\end{tabular}
\end{small}
\end{center}
\caption{Jacobian's block times a vector: number of flops}\label{table_flop}
\end{table}

Below we provide an estimate of a computational effort 
required by the proposed algorithm under mild assumptions: 
\begin{enumerate}
\item $nnz(A)=O(n^2)$, 
\item at Step 3 of Algorithm \ref{GF_algo_upd} a line-search 
first-order method is used to compute an approximate minimizer 
$(U_k, \bar y_k)$ of $\phi_{\mu_k}(U,y)$ such that
$$
\| \nabla \phi_{\mu_k}(U_k,\bar y_k)\| \le \eta_2\, \mu_k.
$$
\end{enumerate}
Taking into account that a line-search first-order method requires in the worst-case
$O(\mu_k^{-2})$ iterations to achieve $\|\nabla \phi (U_k,y_k)\|\le \mu_k$
\cite{grapiglia2017worst}, the computational effort of iteration $k$ 
of Algorithm \ref{GF_algo_upd} is $O(\mu_k^{-2}(n^2 r+m))$ in the worst-case. 
Therefore, when $n$ is large, in the early/intermediate stage of the iterative 
process, this effort is significantly smaller than $O(n^6)$ required 
by a general purpose interior-point solver \cite{AnjosLassere12,deKlerk} 
or $O(n^4)$ needed by the specialized approach for nuclear norm 
minimization \cite{LiuVanden2009}. We stress that this is a worst-case analysis and
in practice we expect to perform less than $O(\mu_k^{-2})$ iterations of the first-order method. 
In case the number of iterations is of the order of $O(n)$   the computational effort per iteration  
of Algorithm \ref{GF_algo_upd}  drops to $O(n^3 r+nm)$.

Apart from all operations listed above the backtracking along $\Delta S$
needs to ensure that $S_k$ is positive definite (Algorithm \ref{GF_algo}, 
Step 4) and this is verified by computing the Cholesky factorization 
of the matrix $S_{k-1}+\alpha_k \Delta S$, for each trial 
steplength $\alpha_k$.  If the dual matrix  is sparse, i.e.  
when matrices $A_i$, $i=1,\ldots,m$ and $C$ share the sparsity 
patterns \cite{VanderbergheAndersen2015}, a sparse Cholesky factor 
is expected. Note that the structure of dual matrix does not change 
during the iterations, hence reordering of $S_0$ 
can be carried out once at the very start of Algorithm \ref{GF_algo_upd}
and then may be reused to compute the Cholesky factorization 
of $S_{k-1}+\alpha_k \Delta S$ at each iteration.
%
%

\subsection{Nonlinear Gauss-Seidel approach}\label{sec:gs}

The crucial step of our interior point framework is the computation 
of an approximate solution of the nonlinear least-squares problem \eqref{least-square}. 
To accomplish the goal, a first-order approach as well as a Gauss-Newton 
method can be used. However, in this latter case the linear algebra phase 
becomes an issue, due to the large dimension of the Jacobian.
Here, we propose  a Nonlinear Gauss-Seidel method.
We also focus on the  linear algebra phase and present a matrix-free implementation   well suited for 
structured constraint matrices as those arising 
in the SDP reformulation of matrix completion problems
\cite{CandesRecht2009}. 
The adopted Nonlinear Gauss-Seidel method 
 to compute $(U_k,\bar y_k)$ at Step 3 of Algorithm \ref{GF_algo_upd}  is detailed in Algorithm \ref{NGS_algo}. 

\begin{algorithm}
\caption{Nonlinear Gauss-Seidel algorithm}\label{NGS_algo}
\begin{algorithmic}[1]
\Require $y_{k-1}$, $U_{k-1}$, $\mu_k$, $\eta_2$ from Algorithm \ref{GF_algo_upd} and $\ell_{max}$.
\State Set $y^0=y_{k-1}$ and $U^0=U_{k-1}$
 \For{$\;\ell=1,2,\dots,\ell_{max}$}
\State 
Set 
\begin{eqnarray*}
 r & = & b - {\cal A} (\mu_k I_n + U^\ell (U^\ell)^T) \\
 R & = & \mu_k I_n - (\mu_k I_n + U^\ell (U^\ell)^T) (C -  {\cal A}^T y^\ell) 
\end{eqnarray*}
 \State Compute a Gauss-Newton step $\Delta U$ for $$\min_U \phi_{\mu_k}(U,y^\ell),$$ 
\hspace*{15pt} that is, solve the linear system 
\begin{equation}\label{sis1}
 [J_{11}^TJ_{11}+ J_{21}^TJ_{21}] vec(\Delta U) = [J_{11}^T J_{21}^T] [r; vec(R)] 
\end{equation}
   \hspace*{15pt} and update $U^{\ell+1}=U^\ell+\Delta U$ 
   and $R = \mu_k I_n - (\mu_k I_n + U^{\ell+1} (U^{\ell+1})^T) (C - {\cal A}^T y^\ell)$
				
 \State Compute a Gauss-Newton step $\Delta y$ for $$\min_y \phi_{\mu_k}(U^{\ell+1},y)$$  
\hspace*{15pt} that is, solve the linear system
 \begin{equation}\label{sis2}
J_{22}^TJ_{22} \Delta y = J_{22}^T vec(R^{})
\end{equation}
          \hspace*{15pt}and update $y^{\ell+1}=y^\ell+\Delta y$.
 \If{ {$\|\nabla \phi_{\mu_k}(U^{\ell+1},y^{\ell+1})\|
\le \eta_2 \mu_k$}}
 \State {\bf return } $U_k=U^{\ell+1}$ and $\bar y_k=y^{\ell+1}$ to Algorithm \ref{GF_algo_upd}.
 \EndIf
 \EndFor
\end{algorithmic}
\end{algorithm}
The computational bottleneck of the procedure given in Algorithm \ref{NGS_algo} 
is the solution of the linear systems \eqref{sis1} and \eqref{sis2}. 
Due to their large dimensions we use a CG-like approach. 
The coefficient matrix in \eqref{sis1} takes the form:
$$
J_{11}^TJ_{11}+ J_{21}^TJ_{21}=Q_k^TA^TA Q_k +
    Q_k^T (S_k^2 \otimes I_n) Q_k = Q_k^T (A^TA + (S_k^2 \otimes I_n)) Q_k 
    \in \IR^{nr \times nr},
$$
and it is positive semidefinite as $Q_k$ may be rank deficient. 
We can apply CG to \eqref{sis1} which is known to converge to the minimum 
norm solution if starting from the null approximation \cite{hestenes75}.
Letting $\bar v \in \IR^{nr}$ be the unitary eigenvector associated 
to the maximum eigenvalue of $Q_k^T (A^TA + (S_k^2 \otimes I_n)) Q_k$ 
and $\bar w=Q_k \bar v$ we have:
\begin{eqnarray*}
\lambda_{max}(Q_k^T (A^TA + (S_k^2 \otimes I_n)) Q_k) 
& =   & \bar w^T (A^TA + (S_k^2 \otimes I_n))) \bar w\\
& \le & \lambda_{max} (A^TA + (S_k^2 \otimes I_n))) \|\bar w\|^2\\
& \le & ((\sigma_{max}(A))^2+(\lambda_{max}(S_k))^2)(\sigma_{max}(Q_k))^2.
\end{eqnarray*}
Moreover, using \ref{Qk} we derive the following bound
\begin{eqnarray*}
\sigma_{max}(Q_k) &=& \sigma_{max}((U_k\otimes I_n) + (I_n \otimes U_k) \Pi_{nr})\\
&\le&  \sigma_{max}(U\otimes I_n) + \sigma_{max} ((I_n \otimes U) \Pi_{nr})\\
&\le&  2 \sigma_{max}(U), 
\end{eqnarray*}
as $\sigma_{max}(U\otimes I_n) =\sigma_{max}(U)$ 
and $\sigma_{max}((U\otimes I_n) )\Pi_{nr})\le\sigma_{max}(U)$.
Since both the maximum eigenvalue of $S_k$ and the maximum singular value 
of $U_k$ are expected to stay bounded from above, we conclude that the 
maximum eigenvalue of $J_{11}^TJ_{11}+ J_{21}^TJ_{21}$ remains bounded.
%
%
The smallest nonzero eigenvalue may go to zero at the same speed 
as $\mu_k^2$.
%
%
However,
 in case of SDP reformulation of matrix completion problems, 
the term $A^TA$ acts as a regularization  term and the smallest nonzero 
eigenvalue of $J_{11}^TJ_{11}+ J_{21}^TJ_{21}$ remains bounded away 
from $\mu_k$ also in the later iterations of the interior point 
algorithm. We will report on this later on,  in the numerical 
results section (see Figure \ref{fig:eig}).

Let us now consider system \eqref{sis2}. 
The coefficient matrix takes the form
\begin{equation}\label{norm-sis-2}
J_{22}^T J_{22} = A (I_n \otimes X_k^2) A^T \in \IR^{m\times m},
\end{equation}
and it is positive definite. We repeat the reasoning applied 
earlier to $J_{11}^TJ_{11}+ J_{21}^TJ_{21}$ and conclude that
$$
\lambda_{max}(J_{22}^T J_{22})\le (\lambda_{max}(X_k))^2(\sigma_{max}(A))^2.
$$
Analogously we have 
$$
\lambda_{min}(J_{22}^T J_{22})\ge (\lambda_{min}(X_k))^2(\sigma_{min}(A))^2.
$$
Taking into account that $r$ eigenvalues of $X_k$ do not depend on $\mu_k$ 
while the remaining are equal to $\mu_k$, we conclude that the condition 
number of $J_{22}^T J_{22}$ increases as ${\cal O}(1/\mu_k^2)$. 
In the next subsection we will show how this matrix can be preconditioned.

\subsection{Preconditioning $J_{22}^TJ_{22}$}
In this subsection we assume that  matrix  $AA^T$ is sparse and easy to invert.
At this regard we underline that in SDP reformulation of matrix-completion problems  
matrices $A_i$ have a very special structure that yields  $AA^T=\frac{1}{2}I_m$.

Note that
substituting $X_k = \mu_k I_n + U_k U_k^T$  in \eqref{norm-sis-2}  we get
\begin{eqnarray}
A (I_n \otimes X_k^2) A^T & = & A ( \mu_k^2 I_{n^2} + 2 \mu_k I_n \otimes U_kU_k^T 
                                 + I_n \otimes (U_kU_k^T)^2 ) A^T \\
& = &  \mu_k^2 AA^T + 2 \mu_k A (I_n \otimes U_kU_k^T) A^T + A(I_n \otimes (U_kU_k^T)^2) A^T \\
& = &  \mu_k^2 AA^T + 2 \mu_k A (I_n \otimes U_kU_k^T) A^T + \nonumber\\
&   &  A (I_n \otimes (U_kU_k^T))(I_n \otimes (U_kU_k^T)) A^T.\end{eqnarray}
Let us consider a preconditioner $P_k$ of the form 
\begin{equation}
\label{prec}
P_k = \mu_k {AA^T } + Z_k  Z_k^T,
\end{equation}
with 
\begin{equation}
\label{z1}
Z_k = A(I_n \otimes (U_kU_k^T)) \in \IR^{m \times n^2}.
\end{equation}
This choice is motivated by the fact that we discard the term 
$I_n \otimes U_kU_k^T$ from the term $2 \mu_k A (I_n \otimes U_kU_k^T) A^T$ 
in the expression of $J_{22}^T J_{22}$.
%
%
In fact, we use the approximation
$$
 \mu_k^2 AA^T + 2 \mu_k A (I_n \otimes U_kU_k^T) A^T\approx \mu_k AA^T. 
$$
A similar idea is used in \cite{zhang2017}. An alternative choice involves matrix $Z_k$ of a smaller dimension 
\begin{equation}
\label{z2}
Z_k = A(I_n \otimes U_k) \in \IR^{m \times nr}.
\end{equation}
This corresponds to introducing a further approximation
$$
 A(I_n \otimes (U_kU_k^T))(I_n \otimes (U_kU_k^T)) A^T\approx  A(I_n \otimes (U_kU_k^T)) A^T.
$$
%
We will analyze spectral properties of the matrix $J_{22}^T J_{22}$ 
preconditioned with $P_k$ defined in (\ref{prec}) with $Z_k$ given in (\ref{z1}).

\begin{theorem}
Let $P_k$ be given in \eqref{prec} with $Z_k$ given in \eqref{z1} 
and $\sigma_{min}(A)$ and $\sigma_{max}(A)$ denote the minimum 
and maximum singular values of $A$, respectively. The eigenvalues 
of the preconditioned matrix $P_k^{-1/2}(A (I \otimes X_k^2) A^T)P_k^{-1/2}$ 
belong to the interval $(1+\xi_1,1+\xi_2)$ 
where $\xi_1$ and $\xi_2$ have the following forms:
$$
\xi_1=\frac{\mu_k(\mu_k-1)(\sigma_{\min}(A))^2}
           {(\sigma_{\max}(A))^2(\mu_k+(\lambda_{\max}(U_kU_k^T))^2)}
$$
and 
$$
\xi_2=\frac{(\sigma_{max}(A))^2(\mu_k+\lambda_{\max} (U_kU_k^T))}
           {(\sigma_{\min}(A))^2}.
$$
\end{theorem}

\begin{proof}
Note that
$$
A (I \otimes X_k^2) A^T =  \mu_k(\mu_k-1) AA^T + 2 \mu_k A (I_n \otimes U_kU_k^T) A^T + P_k.
$$
Then,
$$
P_k^{-1/2}(A (I \otimes X_k^2) A^T)P_k^{-1/2} 
= I + \mu_k P_k^{-1/2}((\mu_k-1)AA^T+2 A (I_n \otimes U_kU_k^T) A^T) P_k^{-1/2}.
$$
Let us denote with $\lambda_M$ and $\lambda_m$ the largest and 
the smallest eigenvalues of matrix \\ 
$P_k^{-1/2}((\mu_k-1)AA^T+2 A (I_n \otimes U_kU_k^T) A^T) P_k^{-1/2}$, 
respectively.
From (\ref{prec}) we deduce  
$$
\lambda_{\min}(P_k) \ge \mu_k (\sigma_{\min}(A))^2 
$$ 
and 
$$
\lambda_{\max}(A (I_n \otimes U_kU_k^T) A^T) \le (\sigma_{max}(A))^2 \lambda_{\max} (U_kU_k^T).
$$
Then, using the Weyl inequality we obtain
$$
\lambda_M \le \frac{(\sigma_{max}(A))^2(\mu_k + \lambda_{\max} (U_kU_k^T))}
            {\mu_k (\sigma_{\min}(A))^2}.
$$
Moreover, 
$$
\lambda_{\min}(P_k^{-1/2} AA^TP_k^{-1/2})=\frac{1}{\lambda_{\max}(P_k^{1/2} (AA^T)^{-1}P_k^{1/2})}
\ge \frac{(\sigma_{\min} (A))^2}{\|P_k\|_2}. 
$$
Then, noting that $\|P_k\|_2 \le (\sigma_{\max} (A))^2( \mu_k +(\lambda_{\max}(U_kU_k^T))^2)$,
we have 
$$
\lambda_m \ge \frac{ (\mu_k-1) (\sigma_{\min}(A))^2}
                     {(\sigma_{\max} (A))^2 (\mu_k+(\lambda_{\max}(U_kU_k^T))^2)}.
$$
Consequently, the eigenvalues of the preconditioned 
matrix $P_k^{-1/2}(A (I \otimes X^2) A^T)P_k^{-1/2}$
belong to the interval $(1+\xi_1,1+\xi_2)$, and the theorem follows.
\end{proof}

Note that from the result above, as $\mu_k$ approaches zero, the minimum eigenvalue 
of the preconditioned matrix goes to one 
and the maximum remains bounded.

The application of $P_k$ to a vector $d$, needed at each CG iteration,
can be performed through the solution of  the $(m+nq) \times (m+nq)$
sparse augmented system:
\begin{equation}
\label{aug}
\begin{bmatrix}
\mu_k AA^T & Z_k\\
Z_k^T & -I_{nr} 
\end{bmatrix} 
\begin{bmatrix}
  u \\
  v
\end{bmatrix}
  =
\begin{bmatrix}
  d \\
  0
\end{bmatrix}.
\end{equation}
where if $Z_k$ is given by (\ref{z1}) $q=n$,
while $q=r$ in case (\ref{z2}).
In order to recover the vector $u=P_k^{-1}d$, we can solve the linear system 
\begin{equation}\label{sis_precon}
  (I_{nr} + Z_k^T ( \mu _kA A^T )^{-1} Z_k) v = Z_k^T ( \mu_k A A^T )^{-1} d,
\end{equation}
and  compute $u$ as follows
$$
  u = ( \mu_k A A^T )^{-1} (d - Z_k v).
$$
This process involves the inversion of $A A^T$ which can be done once 
at the beginning of the iterative process, and the solution of a linear 
system with matrix 
$$E_k =  I+Z_k^T(\mu_k AA^T)^{-1}Z_k.$$ 
Note that $E_k$ has dimension $n^2 \times n^2$ in case of choice (\ref{z1}) 
and dimension $nr \times nr$ in case of choice (\ref{z2}). Then, its inversion 
is impractical in case (\ref{z1}). On the other hand, using (\ref{z2}) 
we can approximately solve \eqref{sis_precon} using a CG-like solver. 

At this regard, observe that the entries of $E_k$ decrease when 
far away from the main diagonal and $E_k$ can be preconditioned 
by its block-diagonal part, that is by 
\begin{equation}\label{prec-prec}
M_k = I_{nr} + {\cal B} (Z_k^T ( \mu_k A A^T )^{-1} Z_k),
\end{equation}
where ${\cal B}$ is the operator that extracts from a matrix 
$nr \times nr$ its block diagonal part with $n$ diagonal blocks 
of size $r\times r$.

\section{SDP reformulation of matrix completion problems}\label{sec:mc}
We consider the problem of recovering a low-rank data matrix 
$B \in \IR^{\hat n \times \hat n}$ from a sampling of its 
entries \cite{CandesRecht2009}, that is the so called 
{\em matrix completion} problem. The problem can be stated as
\begin{equation}\label{matcompl}
\begin{array}{ll}
\min & rank(\bar X) \\[0.2cm]
\mbox{s.t. } & \bar X_{\Omega} = B_{\Omega},  
\end{array}
\end{equation} 
where $\Omega$ is the set of locations corresponding to the observed 
entries of $B$ and the equality is meant element-wise, 
that is $X_{s,t} = B_{s,t}, \mbox{ for all } (s,t) \in \Omega$.
Let $m$ be the cardinality of $\Omega$ and $r$ be the rank of $B$.

A popular convex relaxation of the problem \cite{CandesRecht2009} 
consists in finding the minimum nuclear norm of $\bar X$ that satisfies 
the linear constraints in (\ref{matcompl}), that is, solving 
the following heuristic optimization
\begin{equation}\label{matcompl-nn}
\begin{array}{ll}
\min & \|\bar X \|_* \\[0.2cm]
\mbox{s.t. } & \bar X_{\Omega} = B_{\Omega}, 
\end{array}
\end{equation} 
where the nuclear norm $\|\cdot \|_*$ of $\bar X$ is defined 
as the sum of its singular values.


Cand\`es and Recht proved in \cite{CandesRecht2009} that if $\Omega$ 
is sampled uniformly at random among all subset of cardinality $m$ 
%
%
then with large probability, the unique solution to (\ref{matcompl-nn}) 
is exactly $B$, provided that the number of samples obeys 
$m \ge C \hat n^{5/4} r \log \hat n$, for some positive numerical constant $C$.
In other words, problem (\ref{matcompl-nn}) is ``formally equivalent'' 
to problem (\ref{matcompl}). 
Let
\begin{equation}\label{x_mc}
X=\begin{bmatrix}
W_1 & \bar X \\
\bar X^T & W_2
\end{bmatrix},
\end{equation}
where $\bar X \in \IR^{\hat n\times \hat n}$ is the matrix 
to be recovered and $W_1, W_2 \in S\IR^{\hat n\times \hat n}$. 
Then problem (\ref{matcompl-nn}) can be stated as an SDP 
of the form (\ref{sdp_primal}) as follows  
\begin{equation}\label{matcompl-sdp}
\begin{array}{ll}
\min & \frac{1}{2} I \bullet X \\[0.2cm]
\mbox{s.t. } & \begin{bmatrix}
0 & \Theta_{st} \\
\Theta_{st}^T & 0
\end{bmatrix} \bullet X = B_{(s,t)}, \quad (s,t) \in \Omega \\
[0.5cm]
& X\succeq 0, 
\end{array}
\end{equation} 
where for each $(s,t) \in \Omega$ the matrix
$\Theta_{st} \in \IR^{\hat n \times \hat n}$ is defined element-wise 
for $k,l = 1, \dots, \hat n$ as
$$
(\Theta_{st} )_{kl}= \left \{  \begin{array}{ll}
1/2 & \mbox{ if } (k,l) = (s,t) \\
0 & \mbox{ otherwise, }
\end{array}
\right.
$$
see \cite{RechtFazelParrilo10}. 
We observe that primal variable $X$ takes the form \eqref{x_mc} 
with $n= 2\hat n$, the symmetric matrix $C$ in the objective 
of (\ref{sdp_primal}) is a scaled identity matrix of dimension
$n \times n$. The vector $b\in \IR^m$ is defined by the known 
elements of $B$ and,  for $i=1,\ldots,m$, each constraint matrix $A_i$,  
corresponds to the known elements of $B$ stored in $b_i$. 
Matrices $A_i$ have a very special structure that yields nice 
properties in the packed matrix $A$. 
 Since every constraint matrix 
has merely two nonzero entries the resulting matrix $A$ 
has $2m$ nonzero elements and its density is equal to $2n^{-2}$.
Moreover,  $AA^T=\frac{1}{2}I_m$ 
and $\|{\cal A} (I_n)\|_2 = 0$.

We now discuss the relationship between a rank $r$ solution $\bar X$ 
of problem (\ref{matcompl-nn}) and a rank $r$ solution $X$ 
of problem (\ref{matcompl-sdp}).

\begin{proposition} 
If $X$  of the form 
$\begin{bmatrix} 
W_1 & \bar X \\
\bar X^T & W_2
\end{bmatrix}$ 
with $\bar X \in \IR^{\hat n\times \hat n}$ 
and $W_1, W_2 \in S\IR^{\hat n\times \hat n}$ 
has rank $r$, then $\bar X$ has rank $r$. \\
Vice-versa, if $\bar X$ has rank $r$ 
with $\bar X \in \IR^{\hat n\times \hat n}$, 
then there exist $W_1, W_2 \in S\IR^{\hat n\times \hat n}$ such that
$\begin{bmatrix} 
W_1 & \bar X \\
\bar X^T & W_2
\end{bmatrix}$ has rank $r$.
\end{proposition}
\begin{proof}
Let $X = Q \Sigma Q^T$ with $Q \in \IR^{2 \hat n\times r}$ and
$\Sigma = \IR^{r \times r}$ be the singular value decomposition 
(SVD) of $X$. Let $Q$ be partitioned by 
$Q = \begin{bmatrix} 
Q_1 \\
Q_2
\end{bmatrix}$ 
with $Q_1, Q_2 \in \IR^{ \hat n\times r}$.
Then
$$
X = \begin{bmatrix} 
Q_1 \\
Q_2
\end{bmatrix} \Sigma  \begin{bmatrix} 
Q_1^T & Q_2^T
\end{bmatrix} = \begin{bmatrix} 
Q_1 \Sigma Q_1^T & Q_1 \Sigma Q_2^T    \\
Q_2 \Sigma Q_1^T & Q_2 \Sigma Q_2^T 
\end{bmatrix},
$$
that is $\bar X=Q_1\Sigma Q_2^T$ has rank $r$.

To prove the second part of the proposition, 
let $\bar X = Q \Sigma V^T$ with $Q, V\in \IR^{ \hat n\times r}$ 
and $\Sigma = \IR^{r\times r}$ be the SVD factorization of $\bar X$.  
We get the proposition by defining $W_1 = Q \Sigma Q^T$ 
and $W_2 = V \Sigma V^T$ and obtaining 
$X=\begin{bmatrix} 
Q\\
V
\end{bmatrix} \Sigma \begin{bmatrix} 
Q^T & V^T
\end{bmatrix} .
$
\end{proof}
\begin{corollary}\label{cor}
Let $X$ structured as 
$\begin{bmatrix} 
W_1 & \bar X \\
\bar X^T & W_2
\end{bmatrix}$ 
with $\bar X \in \IR^{\hat n\times \hat n}$ and 
$W_1, W_2 \in S\IR^{\hat n\times \hat n}$. Assume that $X$ has the form
$$
X = UU^T + \mu I,
$$
with $U \in \IR^{n\times r}$ full column rank and $\mu \in \IR$, 
then $\bar X$ has rank r.
\end{corollary}

\begin{proposition}\label{lem:sol}
If $X$ is a rank $r$ solution of (\ref{matcompl-sdp}), then $\bar X$ 
is a rank $r$ solution of (\ref{matcompl-nn}).
Vice-versa, if $\bar X$ is a rank $r$ solution of (\ref{matcompl-nn}), 
then (\ref{matcompl-sdp}) admits a rank $r$ solution.
\end{proposition}
\begin{proof}
The first statement follows from the equivalence between problems 
(\ref{matcompl-sdp}) and (\ref{matcompl-nn}) 
\cite[Lemma 1]{FazelHindiBoyd01}.

Let $\bar X$ be a rank $r$ optimal solution of (\ref{matcompl-nn}), 
$t^* = \|\bar X\|_*$ and $Q \Sigma V^T$,
with $Q, V\in \IR^{ \hat n\times r}$ and $\Sigma \in  \IR^{r\times r}$, 
be the SVD decomposition of $\bar X$.  
Let us define 
$X = \begin{bmatrix} 
W_1 & \bar X \\
\bar X^T & W_2
\end{bmatrix}$ 
with $W_1 = Q \Sigma Q^T$ and $W_2 = V \Sigma V^T$. 
Then $X$ solves (\ref{matcompl-sdp}).
In fact, $X$ is positive semidefinite and 
$\frac{1}{2}I \bullet X = \frac{1}{2}(Trace(W_1)+Trace(W_2)) = \|\bar X\|_* = t^*$.
This implies that $t^*$ is the optimal value of (\ref{matcompl-sdp}). 
In fact, if we had $Y$ such that 
$$
\begin{bmatrix}
0 & \Theta_{st}  \\
\Theta_{st}^T & 0
\end{bmatrix} 
\bullet Y = B_{(s,t)}, \quad (s,t) \in \Omega \quad\quad
Y\succeq 0
$$
and $\frac{1}{2}I \bullet Y \le t^\star$, 
then by \cite[Lemma 1]{FazelHindiBoyd01} 
there would exist $\bar Y$ such that $ \|\bar Y\|_* < t^*$, 
that is $ \|\bar Y\|_* <\|\bar X\|_*=t^*$.
This is a contradiction as we assumed that $t^*$ is the optimal 
value of (\ref{matcompl-nn}). 
\end{proof}

{\em Remark.}  
Assuming that a rank $r$ solution to (\ref{matcompl-nn}) exists, 
the above analysis justifies the application of our algorithm to search 
for a rank $r$ solution of the SDP reformulation (\ref{matcompl-sdp}) 
of (\ref{matcompl-nn}).
We also observe that at each iteration our algorithm computes 
an approximation $X_k$ of the form $X_k=U_kU_k^T+\mu_k I_n$ 
with $U_k\in  \IR^{n\times r}$ and $\mu_k>0$.
Then, if at each iteration $U_k$ is full column rank, 
by Corollary \ref{cor}, it follows that we generate a sequence 
$\{\bar X_k\}$ such that $\bar X_k$ has exactly rank $r$ at each 
iteration $k$ and it approaches a solution of (\ref{matcompl-nn}).  

Finally, let us observe that $m<{\hat n}^2=n^2/4$ {\color{black} and $nnz(A)=2m<n^2/2$.}
Then, by the analysis carried out in Subsection 4.1 each evaluation 
of the gradient of $\phi_{\mu_k}$ amounts to $O(n^2r)$ flops and 
assuming to use a 
first-order method at each iteration to compute 
$(U_k,\bar y_k)$, in the worst-case each iteration of our method requires $O(\mu_k^{-2}n^2r)$ flops.

\section{Numerical experiments on matrix completion problems}\label{sec:exp}

We consider an application to matrix completion problems by solving 
(\ref{matcompl-sdp}) with our relaxed Interior Point algorithm 
for Low-Rank SDPs (IPLR), described in Algorithm \ref{GF_algo_upd}. 
IPLR has been implemented using Matlab (R2018b) and all experiments 
have been carried out on Intel Core i5 CPU 1.3 GHz with 8 GB RAM.
Parameters in Algorithm \ref{GF_algo_upd} have been chosen as follows:
$$
\mu_0 = 1,\ \sigma = 0.5,\ \eta_1 = 0.9,\ \eta_2 = \sqrt{n},
$$
while the starting dual feasible approximation has been chosen 
as $y_0=0, S_0=\frac{1}{2}I_n$ and $U_0$ is defined 
by the first $r$ columns of the identity matrix $I_n$.

{\color{black}
We considered two implementations of IPLR which differ with the strategy used to
find a minimizer of $\phi_{\mu_k}(U,y)$ (Line 3 of Algorithm \ref{GF_algo_upd}).}

Let {\sc IPLR-GS} denote the implementation of IPLR where 
the Gauss-Seidel strategy described in Algorithm \ref{NGS_algo} 
is used to find a minimizer of $\phi_{\mu_k}(U,y)$. 
We impose a maximum number of 5 $\ell$-iterations and use 
the (possibly) preconditioned conjugate gradient method 
to solve the linear systems (\ref{sis1}) and (\ref{sis2}).
We set a maximum of 100 CG iterations and the tolerance $10^{-6}$ 
on the relative residual of the linear systems. System (\ref{sis1}) 
is solved with unpreconditioned CG. Regarding \eqref{sis2}, 
for the sake of comparison, we report in the next section 
statistics using unpreconditioned 
CG and CG employing  the preconditioner defined by (\ref{prec}) 
and (\ref{z2}). In this latter case the action of the preconditioner 
has been implemented through the augmented system \eqref{aug}, 
following the procedure outlined at the end of Section 5.
The linear system \eqref{sis_precon} has been solved by preconditioned CG, 
with preconditioner \eqref{prec-prec} allowing a maximum of 100 CG 
iterations and using a tolerance $10^{-8}$.
In fact, the linear system  (\ref{sis2}) along the IPLR iterations 
becomes ill-conditioned and the application of the preconditioner 
needs to be performed with high accuracy. 
We will refer to the resulting method as {\sc IPLR-GS\_P}.

As an alternative implementation to {\sc IPLR-GS,} we considered 
the use of a first-order approach to perform the minimization at Line 3 
of Algorithm \ref{GF_algo_upd}. We implemented the Barzilai-Borwein 
method \cite{BarzilaiBorwein88,raydan1997barzilai} with a non-monotone 
line-search following \cite[Algorithm 1]{Serafino2018} and using parameter 
values as suggested therein. The Barzilai-Borwein method iterates 
until $\|\nabla \phi_{\mu_k}(U_k,y_k) \|\le \min(10^{-3}, \mu_k)$ 
or a maximum of 300 iterations is reached. 
We refer to the resulting implementation as {\sc IPLR-BB}.

{\color{black}
The recent literature for the solution of matrix completion problems 
is very rich and there exist many algorithms finely tailored for such 
problems, see e.g. 
\cite{CaiCandesShen2010,chen2012matrix,optspace1,admira,lin2010augmented,fpca,toh2010accelerated,xu2012alternating} 
just to name a few. Among these, we chose the {\sc OptSpace} algorithm 
proposed in \cite{optspace1,optspace2} as a reference algorithm in the forthcoming tests. 
In fact, {\sc OptSpace} compares favourably \cite{optspace2} with the state-of-art solvers 
such as SVT \cite{CaiCandesShen2010}, ADMiRA \cite{admira} and FPCA \cite{fpca}
and its Matlab implementation is publicly available online 
\footnote{{\sc OptSpace}: \url{http://swoh.web.engr.illinois.edu/software/optspace/code.html}.}.
{\sc OptSpace} is a first-order algorithm. Assuming the known solution rank $r$, 
it first generates a good starting guess by computing the truncated SVD 
(of rank $r$) of a suitable sparsification of the available data $B_{\Omega}$ 
and then uses a gradient-type procedure in order to minimize the error 
$\|B-Q\Sigma V^T\|_F$ where $Q,\Sigma, V$ are the SVD factors of the current 
solution approximation. Since $Q$ and $V$ are orthonormal matrices,
the minimization in these variables is performed over the Cartesian product 
of Grassmann manifolds, while minimization in $\Sigma$ is computed exactly 
in $\mathbb{R}^{r\times r}$.
In \cite{optspace2}, {\sc OptSpace} has been equipped with two strategies 
to accommodate the unknown solution rank: the first strategy aims 
at finding a split in the eigenvalue distribution of the sparsified 
(``trimmed'') matrix and on accurate approximation of its singular values and 
the corresponding singular vectors; the second strategy starts from the singular 
vectors associated with the largest singular value and incrementally searches 
for the next singular vectors. The latter strategy yields the so called 
{\sc Incremental OptSpace} variant, proposed to handle ill-conditioned 
problems whenever an accurate approximation of the singular vector  
corresponding to the smallest singular value is not possible and 
the former strategy fails.

Matlab implementations of {\sc OptSpace} 
and {\sc Incremental OptSpace} have been employed in the next sections. We used 
default parameters  except for the maximum number of iterations. The default value is $50$ and, as reported in the next sessions,  it was  occasionally increased  to improve accuracy in the computed solution.  } 

We perform two sets of experiments: the first aims at validating 
the proposed algorithms and is carried out on randomly generated problems; 
the second is an application of the new algorithms to real data sets.

\subsection{Tests on random matrices}

{\color{black} 
As it is a common practice for a preliminary assessment of new methods, 
in this section we report on the performance of our proposed IPLR 
algorithm on matrices which have been randomly generated. 
We have generated random matrices both with noise and without noise, random 
nearly low-rank matrices and random mildly ill-conditioned matrices 
with and without noise. For the last class of matrices, which we expect
to mimic reasonably well the practical problems, we also report 
the solution statistics obtained with {\sc OptSpace}.}

We have generated $\hat n \times \hat n$ matrices of rank $r$ by sampling 
two $\hat n \times r$ factors $B_L$ and $B_R$ independently, each 
having independently and identically distributed Gaussian entries, 
and setting $B = B_L B_R$. The set of observed entries $\Omega$ 
is sampled uniformly at random among all sets of cardinality $m$.
The matrix  $B$ is declared recovered if the (2,1) block $\bar X$ 
extracted from the solution $X$ of (\ref{matcompl-sdp}), satisfies 
\begin{equation}\label{rec}
\| \bar X - B\|_F / \|B\|_F < 10^{-3},
\end{equation}
see \cite{CandesRecht2009}.

Given $r$, we chose $m$ by setting $m = c r(2\hat n-r)$,  $\hat n=600,700,800,900,1000$. We used 
$c=0.01 \hat n+4$.
These corresponding values of $m$ 
are much lower than the theoretical bound provided by \cite{CandesRecht2009}
and recalled in Section \ref{sec:mc}, but in our experiments they 
were sufficient to recover the sought matrix by IPLR.

In our experiments, the accuracy level in the matrix recovery 
in (\ref{rec}) is always achieved 
by setting $\epsilon = 10^{-4}$ in Algorithm \ref{GF_algo_upd}.

In the forthcoming tables we report: dimensions $n$ and $m$ 
of the resulting SDPs and target rank $r$ of the matrix to be 
recovered; being $X$ and $S$ the computed solution,
the final primal infeasibility  $\|{\cal A}(X)-b\|$, the complementarity 
gap $\|XS-\mu I\|_F$, the error in the solution of the matrix completion 
problem ${\cal E}= \| \bar X - B\|_F /\|B\|_F$, the overall cpu 
time in seconds.

In Tables \ref{tab:0-GS} and  \ref{tab:0-BB} we report statistics 
of {\sc IPLR-GS} and {\sc IPLR-BB}, respectively. We choose
as a starting rank $r$ the rank of the matrix $B$ to be recovered. 
In the last column of Table \ref{tab:0-GS} we report both the overall 
cpu time of {\sc IPLR-GS} without preconditioner (cpu) and with 
preconditioner (cpu\_P) in the solution of (\ref{sis2}). 
The lowest computational time for each problem is indicated in bold.
%
%
\begin{table}[htb!]
\begin{center}
\begin{tabular}{lccccc}
\toprule
                  & \multicolumn{ 5}{c}{\sc IPLR-GS }            \\
\midrule
      rank/$n$/$m$ & $\|{\cal A}(X)-b\|$ & $\|XS-\mu I\|_F$ & $\lambda_{\min}(S)$ &  $\cal E$ &  cpu/cpu\_P  \\
%
%
%
%
%

\midrule
3/1200/35910 &      4E-04 &      1E-03 &     4E-08 &            2E-06 &        229/{\bf 110} \\

4/1200/47840 &      2E-04 &      1E-03 &     4E-08 &           9E-07 &        173/{\bf 99} \\

5/1200/59750 &      4E-05 &      1E-03 &     4E-08 &           1E-07 &        156/{\bf 104} \\

6/1200/71640 &      2E-06 &      1E-03 &     4E-08 &           5E-09 &        219/{\bf 201} \\

7/1200/83510 &      5E-07 &      1E-03 &     4E-08 &           9E-10 &        {\bf 164}/199 \\

8/1200/95360 &      5E-08 &      1E-03 &     4E-08 &           8E-11 &        {\bf 152}/228 \\

\midrule
3/1400/46101 &      3E-04 &      1E-03 &     4E-08 &           1E-06 &        362/{\bf 148} \\

4/1400/61424 &      1E-04 &      1E-03 &     4E-08 &           8E-07 &        352/{\bf 175} \\

5/1400/76725 &      5E-05 &      1E-03 &     4E-08 &           1E-07 &        205/{\bf 151}\\

6/1400/92004 &      7E-06 &      1E-03 &     4E-08 &           1E-08 &        223/{\bf 199} \\

7/1400/107261 &      2E-07 &      1E-03 &     3E-08 &          4E-10 &        {\bf 214}/239 \\

8/1400/122496 &      2E-08 &      1E-03 &     3E-08 &           3E-11 &        {\bf 234}/329 \\

\midrule
3/1600/57492 &      3E-04 &      1E-03 &     3E-08 &          1E-06 &        330/{\bf 168} \\

4/1600/76608 &      1E-04 &      1E-03 &     3E-08 &           4E-07 &        387/{\bf 174} \\

5/1600/95700 &      4E-05 &      1E-03 &     3E-08 &           9E-08 &        433/{\bf 235} \\

6/1600/114768 &      1E-06 &      1E-03 &     3E-08 &          2E-09 &        316/{\bf 226} \\

7/1600/133812 &      2E-07 &      1E-03 &     3E-08 &          2E-10 &        393/{\bf 331} \\

8/1600/152832 &      4E-08 &      1E-03 &     3E-08 &          5E-11 &        {\bf 334}/370 \\
\midrule
3/1800/64692 &      4E-04 &      1E-03 &     3E-08 &          2E-06 &        566/{\bf 259} \\

4/1800/86208 &     3E-04 &      1E-03 &     3E-08 &           7E-07 &        506/{\bf 231} \\

5/1800/107700 &      4E-05 &      1E-03 &     3E-08 &          1E-07 &        465/{\bf 270} \\

6/1800/129168 &      1E-05 &      1E-03 &     3E-08 &          6E-08 &        586/{\bf 364} \\

7/1800/150612 &      8E-07 &      1E-03 &     3E-08 &          3E-9&        606/{\bf 462} \\

8/1800/172032 &      4E-07 &      1E-03 &     3E-08 &          1E-9 &        831/{\bf 795} \\

\midrule
3/2000/83874 &      3E-04 &      1E-03 &     2E-08 &          1E-06 &        599/{\bf 400} \\

4/2000/111776 &     3E-04 &      1E-03 &     2E-08 &           7E-07 &        544/{\bf 365} \\

5/2000/139650 &      1E-05 &      1E-03 &     2E-08 &          3E-08 &        783/{\bf 512} \\

6/2000/167496 &      2E-06 &      1E-03 &     2E-08 &          3E-09 &        601/{\bf 485} \\

7/2000/195314 &      2E-07 &      1E-03 &     2E-08 &          2E-10 &        657/{\bf 594} \\

8/2000/223104 &      2E-08 &      1E-03 &     2E-08 &          4E-11 &        {\bf 627}/669\\

\bottomrule

\end{tabular}  
\caption{{\sc IPLR-GS} on random matrices.\label{tab:0-GS}}
\end{center}
\end{table}

\begin{table}[htb!]
\begin{center}
\begin{tabular}{lccccc}
\toprule
                       & \multicolumn{ 5}{c}{\sc IPLR-BB }                     \\
\midrule
          rank/$n$/$m$ & $\|{\cal A}(X)-b\|$ & $\|XS-\mu I\|_F$ & $\lambda_{\min}(S)$ &  ${\cal E}$ &       cpu  \\
\midrule
3/1200/35910 &      4E-06 &      1E-03 &     4E-08 &          2E-08 &        223 \\

4/1200/47840 &      1E-05 &      1E-03 &     4E-08 &          3E-08 &        186\\

5/1200/59750 &      6E-06 &      1E-03 &     4E-08 &          2E-08 &        235 \\

6/1200/71640 &      8E-06 &      1E-03 &     4E-08 &          1E-08 &        242 \\

7/1200/83510 &      4E-06 &      1E-03 &     4E-08 &          9E-09 &        237 \\

8/1200/95360 &      6E-06 &      1E-03 &     4E-08 &          1E-08 &        223 \\

\midrule
3/1400/46101 &      8E-06 &      1E-03 &     4E-08 &          3E-08 &        402 \\

4/1400/61424 &      2E-06 &      1E-03 &     4E-08 &           8E-08 &       402 \\

5/1400/76725 &      6E-06 &      1E-03 &     4E-08 &           1E-08 &       332 \\

6/1400/92004 &      4E-06 &      1E-03 &     3E-08 &           9E-09 &       403 \\

7/1400/107261 &     2E-06 &      1E-03 &     3E-08 &          4E-09 &        361 \\

8/1400/122496 &     2E-06 &      1E-03 &     3E-08 &          6E-09 &        386 \\

\midrule
3/1600/57492 &      2E-04 &      1E-03 &     3E-08 &          6E-09 &        557 \\

4/1600/76608 &      4E-06 &      1E-03 &     3E-08 &           8E-09 &        620\\

5/1600/95700 &      2E-06 &      1E-03 &     3E-08 &           5E-08 &        506 \\

6/1600/114768 &      2E-06 &      1E-03 &     3E-08 &          3E-09 &        477 \\

7/1600/133812 &      4E-06 &      1E-03 &     3E-08 &          5E-09 &        571 \\

8/1600/152832 &      4E-07 &      1E-03 &     3E-08 &          5E-10 &        600 \\
\bottomrule
3/1800/64692 &      9E-06 &      1E-03 &     3E-08 &          6E-08 &        573 \\

4/1800/86208 &     8E-06 &      1E-03 &     3E-08 &           4E-08 &        906 \\

5/1800/107700 &      4E-06 &      1E-03 &     3E-08 &          1E-08 &        784 \\

6/1800/129168 &      2E-06 &      1E-03 &     3E-08 &          6E-09 &        686 \\

7/1800/150612 &      3E-06 &      1E-03 &     3E-08 &          1E-8&        625 \\

8/1800/172032 &      4E-07 &      1E-03 &     3E-08 &          1E-8 &        862 \\

\midrule
3/2000/83874 &      7E-06 &      1E-03 &     3E-08 &          3E-08 &        900 \\

4/2000/111776 &     4E-07 &      1E-03 &     3E-08 &           9E-10 &        1000 \\

5/2000/139650 &      4E-06 &      1E-03 &     3E-08 &          1E-08 &        921 \\

6/2000/167496 &      7E-06 &      1E-03 &     2E-08 &          1E-08 &         900 \\
7/2000/195314 &      3E-07 &      1E-03 &     2E-08 &          3E-10 &       1000 \\
8/2000/223104 &      4E-08 &      1E-03 &     2E-08 &          3E-09 &       931 \\

\bottomrule
\end{tabular}  
\caption{{\sc IPLR-BB} on random matrices.\label{tab:0-BB}}
\end{center}
\end{table}

%
%
As a first comment, we verified that Assumption 1 in Section \ref{sectDBA} 
holds in our experiments. In fact, the method manages to preserve 
positive definiteness of the dual variable and $\alpha_k<1$ 
is taken only in the early stage of the iterative process.

Secondly, we observe that both {\sc IPLR-GS} and {\sc IPLR-BB} provide 
an approximation to the solution of the sought rank; in some runs 
the updating procedure increases the rank, but at the subsequent 
iteration the downdating strategy is activated and the procedure 
comes back to the starting rank $r$.
Moreover, {\sc IPLR-GS} is overall less expensive than {\sc IPLR-BB} 
in terms of cpu time, in particular as $n$ and $m$ increase. In fact, 
the cost of the linear algebra in the {\sc IPLR-GS} framework is 
contained as one/two inner Gauss-Seidel iterations are performed 
at each outer {\sc IPLR-GS} iteration except for the very few initial 
ones where up to five inner Gauss-Seidel iterations are needed.
To give more details of the computational cost of both methods, 
in Table \ref{GS-BB-900} we report some statistics of {\sc IPLR-GS} 
and {\sc IPLR-BB} for $ \hat n=900$, $r=3$ and $8$. 
More precisely we report the average number of inner Gauss-Seidel 
iterations (avr\_GS) and the average number of unpreconditioned CG 
iterations in the solution of \eqref{sis1} (avr\_CG\_1) 
and \eqref{sis2} (avr\_CG\_2) for {\sc IPLR-GS} and the average 
number of BB iterations for {\sc IPLR-BB} (avr\_BB). We notice that 
the solution of SDP problems becomes more demanding as the rank 
increases, but both the number of BB iterations and the number 
of CG iterations are reasonable. 
\begin{table}[h]
\begin{center}
\begin{tabular}{lccc|c}
\toprule
                  & \multicolumn{ 3}{c}{\sc IPLR-GS }   &       {\sc IPLR-BB}       \\
\midrule 
      rank/$n$/$m$ & avr\_GS &avr\_CG\_1  & avr\_CG\_2  &   avr\_BB   \\
      \midrule
3/1800/64692       &   2.1 &   15.3  & 24.2    & 68\\
 8/1800/172032    &  2.0 &  19.5 & 42.2  & 88 \\ 
 \bottomrule

\end{tabular}  
\caption{Statistics of {\sc IPLR-GS} and {\sc IPLR-BB} on random 
         matrix $\hat n=900$, $r=3$ and $8$.\label{GS-BB-900}}
\end{center}
\end{table}

To provide an insight into the linear algebra phase, in Figure \ref{fig:eig} 
we plot the minimum nonzero eigenvalue and the maximum eigenvalue 
of the coefficient matrix of (\ref{sis1}), 
i.e. $Q_k^T (A^TA + (S_k^2 \otimes I_n)) Q_k$. 
We remark that the matrix depends both on the outer iteration $k$ and 
on the inner Gauss-Seidel  iteration $\ell$ and we dropped the index 
$\ell$ to simplify the notation. Eigenvalues are plotted against 
the inner/outer iterations, for $ \hat n=100$, $r=4$ and {\sc IPLR-GS} 
continues until $\mu_k<10^{-7}$. In this run only one inner 
iteration is performed at each outer iteration except for the first 
outer iteration.
We also plot in the left picture of Figure \ref{fig:cg} the number 
of CG iterations versus inner/outer  iterations.
The figures show that the condition number of $Q_k$ and the overall 
behaviour of CG do not depend on $\mu_k$.  Moreover, Table \ref{GS-BB-900} 
shows that unpreconditioned CG is able to reduce the relative residual 
below $10^{-6}$ in a low number of iterations even in the solution 
of larger problems and higher rank. These considerations motivate our 
choice of solving (\ref{sis1}) without employing any preconditioner. 
%
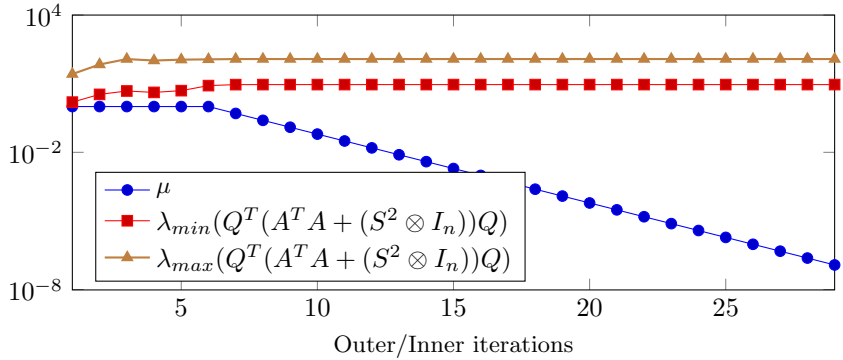
\begin{figure} \centering
		\begin{tikzpicture}
		\begin{semilogyaxis}[width=. 9* \textwidth, 
		   legend pos=south west,
		   legend cell align=left,
		   xlabel = {\small Outer/Inner iterations}, 
		   height = .25 \textheight,
		   xmax = 29,
		   xmin = 1,
		   ymax = 10000
		 ]
		  \addplot+ table[x index = 0, y index = 1] {./eig_deltau.dat};
		  \addplot+ table[x index = 0, y index = 2] {./eig_deltau.dat};
		  \addplot+[mark=triangle*,brown,thick,mark options={ fill=brown}] table[x index = 0, y index = 3] {./eig_deltau.dat};		  
	 \legend{\small $\mu$,$\lambda_{min}(Q^T (A^TA + (S^2 \otimes I_n)) Q)$,$\lambda_{max}(Q^T (A^TA + (S^2 \otimes I_n)) Q)$};
		\end{semilogyaxis}
	\end{tikzpicture}		
  \caption{The minimum nonzero eigenvalue and the maximum eigenvalue 
  of the coefficient matrix of (\ref{sis1}) and $\mu_k$ (semilog scale) 
  versus Outer/Inner {\sc IPLR-GS} iterations. Data: $\hat n=100$, $r=4$.}
\label{fig:eig}
	\end{figure}

We now discuss the effectiveness of the preconditioner $P_k$ given 
in \eqref{prec}, with $Z_k$ given in \eqref{z2}, in the solution 
of \eqref{sis2}. Considering $ \hat n=100$, $r=4$, in Figure \ref{fig:eigsis2} 
we plot the eigenvalue distribution (in percentage) 
of $A (I \otimes X_k^2) A^T$ and $P_k^{-1}(A (I \otimes X_k^2) A^T)$ 
at the first inner iteration of outer {\sc IPLR-GS} iteration 
corresponding to $\mu_k \approx 1.9e\!-\!3$. We again drop the index $\ell$.
We can observe that the condition number of the preconditioned 
matrix is about $1.3e5$, and it is significantly smaller than the condition 
number of the original matrix (about $3.3e10$). The preconditioner succeeded 
both in pushing the smallest eigenvalue away from zero and in reducing 
the largest eigenvalue.  
However, CG converges in a reasonable number of iterations even 
in the unpreconditioned case, despite the large condition number. 
In particular, we can observe in the right picture of Figure \ref{fig:cg} 
that preconditioned CG takes less than five iterations in the last stages 
of {\sc IPLR-GS} and that the most effort is made in the initial
stage of the {\sc IPLR-GS} method; in this phase the preconditioner 
is really effective in reducing the number of CG iterations.
These considerations remain true even for larger values of $\hat n$ 
and $r$ as it is shown in Table \ref{GS-BB-900}.


%

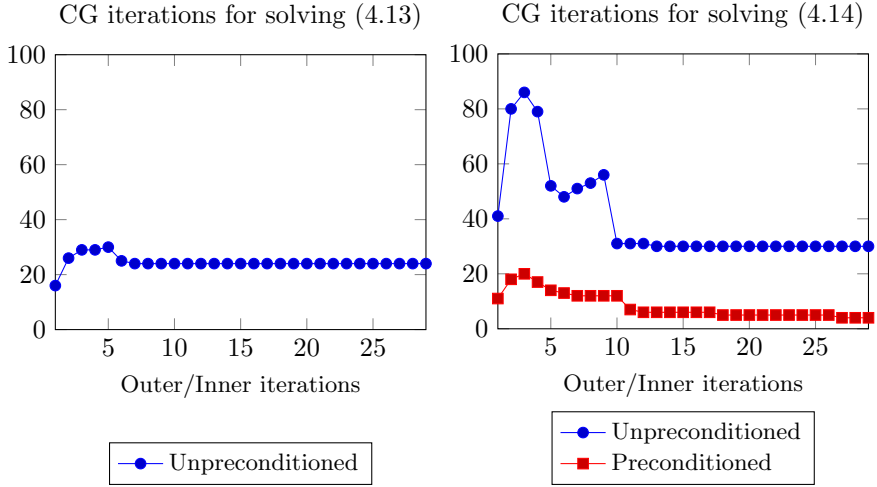
\begin{figure} \centering
\begin{tikzpicture}
		\begin{axis}[width=.5 * \textwidth, 
		   title = {CG iterations for solving (\ref{sis1})}, 
		   legend style={at={(0.5,-0.41)},anchor=north},
		   legend cell align=left,
		   xlabel = {\small Outer/Inner iterations}, 
		   height = .25 \textheight,
			 xmax = 29,
		   xmin = 1,
		   ymax = 100,
		   ymin = 0
		 ]		 
		 \addplot+ table[x index = 0, y index = 1] {./unp_deltau.dat};
	\legend{ \small Unpreconditioned };
		\end{axis}
	\end{tikzpicture}~\begin{tikzpicture}
		\begin{axis}[width=. 5* \textwidth, 
		   title = {CG iterations for solving (\ref{sis2})}, 
		   legend style={at={(0.5,-0.3)},anchor=north},
		   legend cell align=left,
		   xlabel = {\small Outer/Inner iterations}, 
		   height = .25 \textheight,
			 xmax = 29,
		   xmin = 1,
		   ymax = 100,
		   ymin = 0
		 ]
		\foreach \j in {1,2} {
		  \addplot+ table[x index = 0, y index = \j] {./unp_p_deltay.dat};
		}
		\legend{ \small Unpreconditioned, \small Preconditioned  };
		\end{axis}
	\end{tikzpicture}
	\caption{CG iterations for solving systems (\ref{sis1}) with {\sc IPLR-GS} (left) and  
	CG iterations for solving systems (\ref{sis2}) with {\sc IPLR-GS} and {\sc IPLR-GS\_P}   (right). 
	Data: $\hat n=100$, $r=4$.} \label{fig:cg}
	\end{figure}

Focusing on the computational cost of the preconditioner's application, 
we can observe from the cpu times reported in Table \ref{tab:0-GS}, 
that for $r=3,4,5$ the employment of the preconditioner produces 
a great benefit, with savings that vary from $20\%$ to $50\%$.  
Then, the overhead associated to the construction and application 
of the preconditioner is more than compensated by the gains in the number 
of CG iterations. The cost of application of the preconditioner increases 
with $r$ as the dimension of the diagonal blocks of $M_k$ in \eqref{prec-prec} 
increases with $r$. Then, for small value of $\hat n$ and $r=6,7,8$ 
unpreconditioned CG is preferable, while for larger value of $\hat n$ 
the preconditioner is effective in reducing the overall computational 
time for $r\le 7$. This behaviour is summarized in Figure \ref{fig:rankn} 
where we plot the ratio cpu\_P/cpu with respect to dimension $n$ and rank 
(from 3 to 8).


\begin{figure}
\centering
\begin{tikzpicture}
\begin{axis}[width=. 5* \textwidth, 
		   ybar interval,xtick=data,
	xticklabel interval boundaries,%
	ylabel={Percentage},
	xlabel={$\log_{10}(\lambda(A (I \otimes X_k^2) A^T))$},
	x tick label style=
		{rotate=45,anchor=east}]
\addplot+[ybar interval] plot coordinates
                  {(-6,80) (-5,0)  (2,5) (3, 15) (4,0)};
\end{axis}
 \end{tikzpicture}~
\begin{tikzpicture}
 \begin{axis}
[width=. 5* \textwidth, 
		   ybar interval,xtick=data,
	xticklabel interval boundaries,
	ylabel={Percentage},
	xlabel={$\log_{10}(\lambda(P_k^{-1}(A (I \otimes X_k^2) A^T)))$},
	x tick label style=
		{rotate=45,anchor=east},
		]
\addplot+[ybar interval] plot coordinates
		{(-3,80) (-2.5,0)  (2,20) (2.5, 0)};
\end{axis}
\end{tikzpicture}

\caption{Eigenvalue distribution of $A (I \otimes X_k^2) A^T$ (left) 
and $P_k^{-1}(A (I \otimes X_k^2) A^T)$ (right)
at the first inner iteration of outer {\sc IPLR-GS} iteration 
corresponding to $\mu_k\approx 1.9 e-3$
(semilog scale). Data: $\hat n=100$, $r=4$.}\label{fig:eigsis2}
\end{figure}
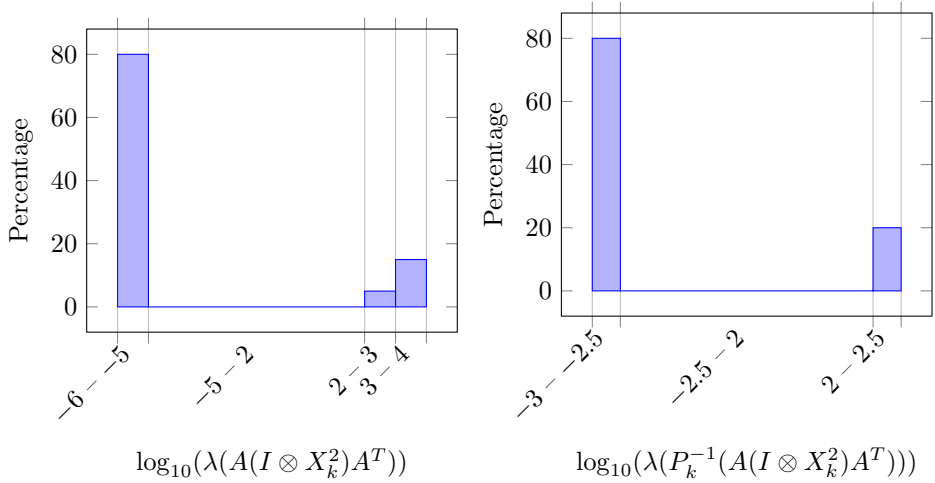
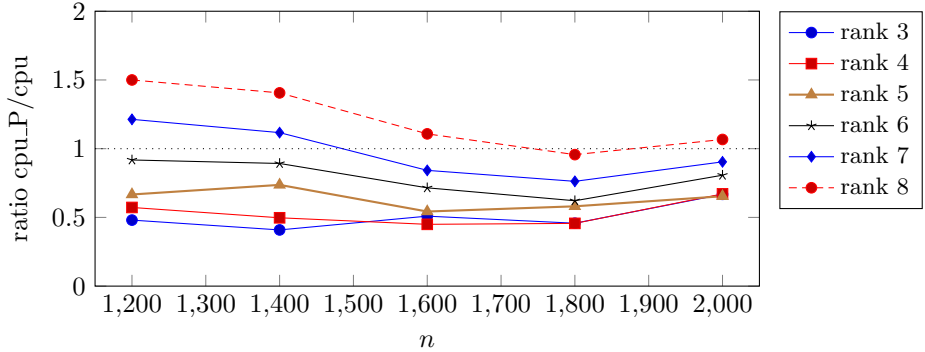
\begin{figure} \centering
		\begin{tikzpicture}
		\begin{axis}[width=. 8* \textwidth, 
		   legend pos=outer north east,
		   legend cell align=left,
		   xlabel = {$n$}, 
		   ylabel = {ratio cpu\_P/cpu},
		   height = .25 \textheight,
		   xmax = 2050,
		   xmin = 1150,
		   ymax = 2,
		   ymin = 0
		 ]
		  \addplot table {./rank_n_prec3.dat};
		  \addplot table {./rank_n_prec4.dat};
		  \addplot[mark=triangle*,brown,thick,mark options={ fill=brown}] table {./rank_n_prec5.dat};
		  \addplot table {./rank_n_prec6.dat};
		  \addplot table {./rank_n_prec7.dat};
		  \addplot table {./rank_n_prec8.dat};
\legend{\small rank $3$, \small  rank $4$,\small rank $5$, \small  rank $6$,\small rank $7$, \small  rank $8$ };
 	         \addplot[dotted]  plot coordinates
 		{(1150,1) (20150,1)};
		\end{axis}
	\end{tikzpicture}		
  \caption{The ratio cpu\_P/cpu as a funcion of dimension $n$ and of the rank
  (data extracted from Table \ref{tab:0-GS}).}
  \label{fig:rankn}
\end{figure}

In the approach proposed in this paper the primal feasibility 
is gradually reached, hence it is also possible to handle data $B_{\Omega}$ 
corrupted by noise. To test how the method behaves in such situations 
we set $\hat B_{(s,t)} = B_{(s,t)} + \eta RD_{(s,t)}$ for any $(s,t) \in \Omega$, 
where $RD_{(s,t)}$ is a random scalar drawn from the standard normal distribution, 
generated by the Matlab function {\tt randn}; $\eta>0$ is the level of noise. 
Then, we solved problem \eqref{matcompl-sdp} using the corrupted 
data $\hat B_{(s,t)}$ to form the vector $b$.
Note that, in this  case $\|{\cal A}(B)-b\|_2\approx \eta \sqrt{m}$. 
In order to take into account the presence of noise we set 
$\epsilon = \max(10^{-4},10^{-1} \eta)$ in Algorithm \ref{GF_algo_upd}. 

Results of these runs are collected in Table \ref{tab::3} where 
we considered $\eta = 10^{-1}$ and started with the target rank $r$. 
In table \ref{tab::3} we also report 
{\color{black} $$RMSE= \| \bar X - B\|_F /\hat n,$$} 
that is the root-mean squared error per entry. 
Note that the root-mean error per entry 
in data $B_{\Omega}$ is of the order of the noise level $10^{-1}$, 
as well as $\|{\cal A}(B)-b\|_2/\sqrt{m}$. Then, we claim to recover 
the matrix with acceptable accuracy, corresponding to an average error 
smaller than the level of noise.

\begin{table}[htb!]
\begin{center}
\begin{tabular}{lccccc}
\toprule
                    & \multicolumn{ 5}{c}{{\sc IPLR-GS\_P}}    \\
\midrule
  rank/$n$/$m$      & $\|{\cal A}(X)-b\|$ & $\|XS-\mu I\|_F$ & $\lambda_{\min}(S)$ &  $ \| \bar X - B\|_F /\hat n$ &       cpu  \\

\midrule

4/1200/47840  &     2E01  &      1E-01   &   6E-06 &      3E-02 &     67 \\

6/1200/71640  &     2E01  &      1E-01   &   6E-06 &      3E-02 &    128 \\

8/1200/95360  &     3E01  &      1E-01   &   5E-06 &      3E-02 &    182 \\

\midrule

4/1600/76608  &     3E01  &      2E-01   &   4E-06 &      3E-02 &    178 \\

6/1600/114768 &     3E01  &      2E-01   &   4E-06 &      3E-02 &    224 \\

8/1600/152832 &     4E01  &      2E-01   &   4E-06 &      3E-02 &    358 \\
\midrule

4/2000/111776 &     3E01  &      2E-01   &   4E-06 &      3E-02 &    259 \\

6/2000/167496 &     4E01  &      2E-01   &   4E-06 &      3E-02 &    373 \\

8/2000/223104 &     4E01  &      2E-01   &   4E-06 &      3E-02 &    543 \\

\bottomrule
\end{tabular}  
\caption{{\sc IPLR-GS\_P}  on noisy matrices  (noise level $\eta = 10^{-1}$). \label{tab::3}}
\end{center}
\end{table}

{\color{black} 
\subsection*{\bf Mildly ill-conditioned problems}
In this subsection we compare the performance of {\sc IPLR\_GS\_P}, {\sc OptSpace} 
and {\sc Incremental OptSpace} on mildly ill-conditioned problems with exact 
and noisy observations.
We first consider exact observation and vary the condition number of the matrix 
that has to be recovered $\kappa$.
We fixed $\hat n=600$ and $r=6$ and, following \cite{optspace2}, generated 
random matrices with a prescribed condition number $\kappa$ and rank $r$ as follows.
Given a random matrix $B$ generated as in the previous subsection, 
let $Q\Sigma V^T$ be its SVD decomposition and $\tilde Q$ and $\tilde V$ 
be the matrices formed by the first $r$ columns of $Q$ and $V$, respectively. 
Then,  we formed  the matrix $\hat B$ that has to be recovered 
as $\hat B=\tilde Q \tilde \Sigma \tilde V^T$, where $\tilde \Sigma$ 
is a $r\times r$ diagonal matrix with diagonal entries equally spaced 
between $\hat n$ and $\hat n/\kappa$.
In Figure \ref{ill-conditioned} we plot the RMSE value against the condition 
number for all the three solvers considered, using the $13\%$ of the observations.
We can observe, as noticed in \cite{optspace2}, that {\sc OptSpace} does not 
manage to recover mildly ill-conditioned matrices while  {\sc Incremental OptSpace} 
improves significantly over {\sc OptSpace}. According to \cite{optspace2}, 
the convergence difficulties of {\sc OptSpace} on these tests has to be 
ascribed to the singular value decomposition of the trimmed matrix needed 
in Step 3 of {\sc OptSpace}. In fact, the singular vector corresponding 
to the smallest singular value cannot be approximated with enough accuracy. 
On the other hand, our approach is more accurate than {\sc Incremental OptSpace} 
and its behaviour only slightly deteriorates as $\kappa$ increases.


\begin{figure} \centering
		\begin{tikzpicture}
		\begin{semilogyaxis}[width= 0.8\textwidth, 
		   legend pos=north east,
		   legend cell align=left,
		   xlabel = {$\kappa$}, 
		   ylabel = {RMSE},
		   height = .25 \textheight,
		   xmax = 100,
		   xmin = 10,
		 ]
		  \addplot table {./ill_conditioned_1.dat};
		  \addplot table {./ill_conditioned_2.dat};
		  \addplot[mark=triangle*,thick, color=brown] table{./ill_conditioned_3.dat};
	\legend{\footnotesize {\sc IPLR-GS\_P}, \footnotesize {\sc OptSpace}, 
                \footnotesize {\sc Incremental OptSpace}};
                \end{semilogyaxis}	
	\end{tikzpicture}			
\caption{{\sc IPLR-GS\_P}, {\sc OptSpace} and {\sc Incremental OptSpace} 
on mildly ill-conditioned matrices (semilog scale $\hat n=600$, $r=6$, $n=1200$, $m=47840$).} 
    \label{ill-conditioned}
    \end{figure}
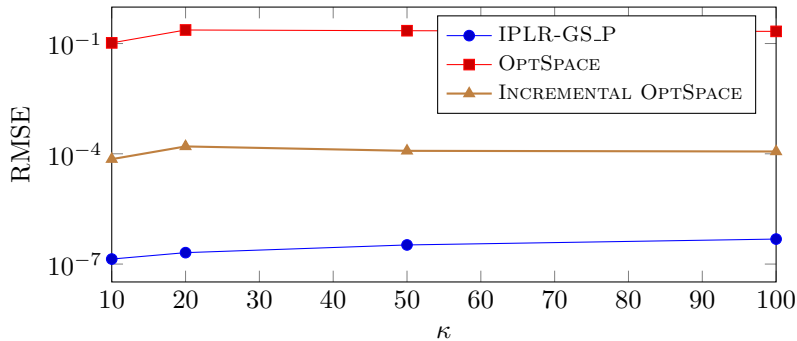

Now, let us focus on the case of noisy observations. We first fixed $\kappa=200$ 
and varied the noise level.  In Figure \ref{noise-ill-conditioned} we plot 
the RMSE value against the noise level for all the three solvers considered, 
using the $20\%$ of observations. Also in this case {\sc IPLR-GS\_P} 
is able to recover the matrix $\hat B$ with acceptable accuracy, corresponding 
to an average error smaller than the level of noise, and outperforms 
both {\sc OptSpace} variants when the noise level is below $0.8$. 
In fact, {\sc OptSpace} managed to recover $\hat B$ only with a corresponding 
$RMSE$ of the order of $10^{-1}$ for any tested noise level, consistent only 
with the larger noise level tested. 


\begin{figure} \centering
		\begin{tikzpicture}
		\begin{semilogyaxis}[width=0.8 \textwidth, 
		   legend pos=south east,
		   legend cell align=left,
		   xlabel = {$\eta$}, 
		   ylabel = {RMSE},
		   height = .25 \textheight,
		   xmax = 1.1,
		   xmin = 0,
		 ]
		  \addplot table {./noise_ill_conditioned_1.dat};
		  \addplot table {./noise_ill_conditioned_2.dat};
		  \addplot[mark=triangle*,thick, color=brown] table {./noise_ill_conditioned_3.dat};
	\legend{\footnotesize  {\sc  IPLR-GS\_P},\footnotesize {\sc OptSpace},\footnotesize {\sc Incremental OptSpace}};\end{semilogyaxis}	
	\end{tikzpicture}			
\caption{{\sc IPLR-GS\_P},   {\sc OptSpace}  and   {\sc Incremental OptSpace} on noisy and mildly  ill-conditioned matrices  (semilog scale, $\kappa=200$, $\hat n=600$, $r=6$, $n=1200$, $m=71640$).} \label{noise-ill-conditioned}
	\end{figure}
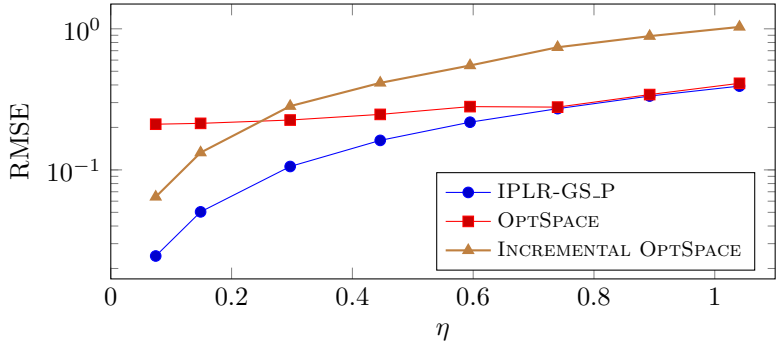


In order to get a better insight into the behaviour of the method on mildly 
ill-conditioned and noisy problems, we fixed $\kappa=100$, noise level 
$\eta=0.3$ and varied the percentage of known entries from $8.3$\% to $50$\%, 
namely we set $m=30000, 45000,$ $60000, 120000, 180000$. 
In Figure \ref{noise-ill-conditioned-density} the value of $RMSE$ 
is plotted against the percentage of known entries. The oracle error 
value $RMSE_{or}=\eta \sqrt{(n1r-r^2)/m}$, given in \cite{candes2010noise} 
is plotted, too. We observe that in our experiments {\sc IPLR-GS\_P} recovers 
the sought matrix with RMSE values always smaller than $1.3RMSE_{or}$, 
despite the condition number of the matrix. This is not the case for {\sc OptSpace} 
and {\sc Incremental OptSpace};  {\sc OptSpace}   can reach a comparable accuracy only 
if the percentage of known entries exceeds $30\%$. 
As expected, for all methods the error decreases as the number of subsampled 
entries increases.


\begin{figure} \centering
		\begin{tikzpicture}
		\begin{axis}[width= 0.8\textwidth, 
		   legend pos=north east,
		   legend cell align=left,
		   xlabel = {percentage of observations}, 
		   ylabel = {RMSE},
		   height = .25 \textheight,
		   xmax = 70,
		   xmin = 0,
		 ]
		  \addplot table {./noise_ill_conditioned_density_1.dat};
		  \addplot table {./noise_ill_conditioned_density_2.dat};
		  \addplot[mark=triangle*,thick, color=brown] table {./noise_ill_conditioned_density_3.dat};
		  \addplot table {./noise_ill_conditioned_density_4.dat};
	\legend{\footnotesize {\sc IPLR-GS\_P}, \footnotesize {\sc OptSpace}, 
                \footnotesize {\sc Incremental OptSpace}, \footnotesize {\sc ORACLE}};
        \end{axis}
	\end{tikzpicture}			
\caption{{\sc IPLR-GS\_P}, {\sc OptSpace} and {\sc Incremental OptSpace} 
on noisy and mildly ill-conditioned matrices, varying the percentage 
of observations ($\kappa=100$, $\hat n=600$, $r=6$, $n=1200$, $\eta=0.3$).} 
        \label{noise-ill-conditioned-density}
	\end{figure}
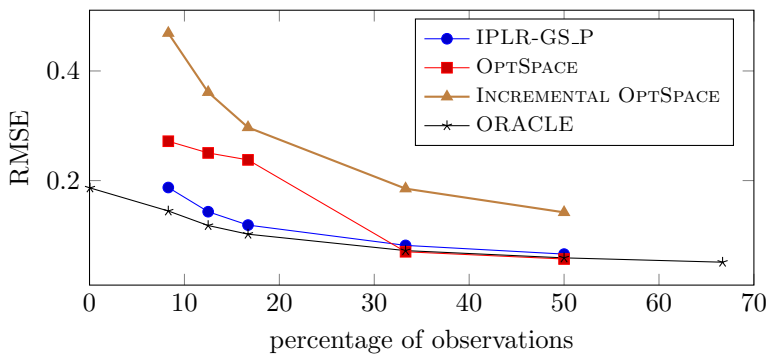

}

{\color{black} 
In summary, for mildly ill-conditioned random matrices our approach 
is more reliable than {\sc OptSpace} and {\sc Incremental OptSpace} 
as the latter algorithms 
might struggle with computing the singular vectors of the sparsified
data matrix accurately, and they cannot deliver precision comparable 
to that of {\sc IPLR}. For the sake of completeness, we remark that 
we have tested {\sc OptSpace} also on the well-conditioned random 
matrices reported in Tables \ref{tab:0-GS}-\ref{tab:0-BB} and \ref{tab::3}. 
On these problems {\sc IPLR} and {\sc OptSpace} provide comparable 
solutions, but as a solver specially designed for matrix-completion 
problems {\sc OptSpace} is generally faster than IPLR. 
}

\noindent
\subsection*{\bf Rank updating}

We now test the effectiveness of the rank updating/downdating strategy 
described in Algorithm \ref{GF_algo_upd}. To this purpose, we run 
{\sc IPLR-GS\_P} starting from $r=1$, with rank increment/decrement 
$\delta_r = 1$ and report the results in Table \ref{tab:1} 
for $\hat n=600,800,1000$. In all runs, the target rank has been 
correctly identified by the updating strategy and the matrix $B$ 
is well-recovered. Runs in italic have been obtained allowing 10 
inner Gauss-Seidel iterations.  In fact, 5 inner Gauss-Seidel 
iterations were not enough to sufficiently  
reduce the residual in \eqref{least-square} and the procedure did not 
terminate with the correct rank. Comparing the values of the cpu time 
in Tables \ref{tab:0-GS} and \ref{tab:1} we observe that the use of rank 
updating strategy increases the overall time; on the other hand, it allows 
to adaptively modify the rank in case a solution of \eqref{matcompl-sdp} 
with the currently attempted rank does not exist.

\begin{table}[htb!]
\begin{center}
\begin{tabular}{lccccc}
\toprule
    &             \multicolumn{ 5}{c}{\sc IPLR-GS\_P }                     \\
\midrule
          rank/$n$/$m$      & $\|{\cal A}(X)-b\|$ & $\|XS-\mu I\|_F$ & $\lambda_{\min}(S)$ &  ${\cal E}$ &       cpu  \\

\midrule
3/1200/35910 &      4E-04 &      1E-03 &     4E-08 &          3E-06 &        161\\

4/1200/47840 &      4E-04 &      1E-03 &     5E-08 &          3E-06 &        206 \\

5/1200/59750 &      5E-05 &      1E-03 &     5E-08 &          3E-07 &        315 \\

6/1200/71640 &      9E-06 &      1E-03 &     4E-08 &          5E-08 &        390 \\

7/1200/83510 &      8E-06 &      1E-03 &     4E-08 &          4E-08 &        494 \\

8/1200/95360 &      {\it 4E-07} &    {\it  1E-03} &    {\it 4E-08} &         {\it 2E-09} &       {\it 746}\\ 
\midrule
3/1600/57492 &      4E-04 &      1E-03 &     3E-08 &          3E-06 &        411 \\

4/1600/76608 &      3E-04 &      1E-03 &     4E-08 &           1E-06 &        488 \\

5/1600/95700 &      7E-05 &      1E-03 &     3E-08 &           3E-07 &        641 \\

6/1600/114768 &     2E-05 &      1E-03 &     3E-08 &          8E-08 &        841 \\

7/1600/133812 &   {\it  4E-07} &      {\it 1E-03} &     {\it 3E-08} &      {\it    1E-09}& {\it       996}\\

8/1600/152832 &      {\it 1E-07} &      {\it 1E-03} &     {\it 3E-08} &          {\it 4E-10} &   1238     \\ 
\midrule
3/2000/83874 &      3E-04 &      1E-03 &     3E-08 &          2E-06 &         566\\ 

4/2000/111776 &     3E-04 &      1E-03 &     3E-08 &           1E-06 &        791\\

5/2000/139650 &      3E-05 &      1E-03 &     3E-08 &          1E-07 &        894\\ 
6/2000/167496 &      9E-06 &      1E-03 &     3E-08 &          1E-08&       1293 \\

7/2000/195314 &      3E-07 &      1E-03 &     3E-08 &          1E-7 &       1809 \\

8/2000/223104 &     {\it  1E-07} &     {\it  1E-03} &   {\it  3E-08} &   {\it       3E-10} &   {\it     2149}\\
\bottomrule
\end{tabular} 
 
\caption{{\sc IPLR-GS\_P} on random matrices starting with $r=1$. \label{tab:1}}
\end{center}
\end{table}

The typical updating behaviour is illustrated in Figure \ref{fig:rankupd}
where we started with rank 1 and reached the target rank 5.
In the first eight iterations a solution of the current rank does not 
exist and therefore the procedure does not manage to reduce the primal 
infeasibility as expected. Then, the rank is increased.
At iteration 9 the correct rank has been detected and the primal infeasibility 
drops down. Interestingly, the method attempted rank 6 at iteration 13, but 
quickly corrected itself and returned to rank 5 which was the right one.

   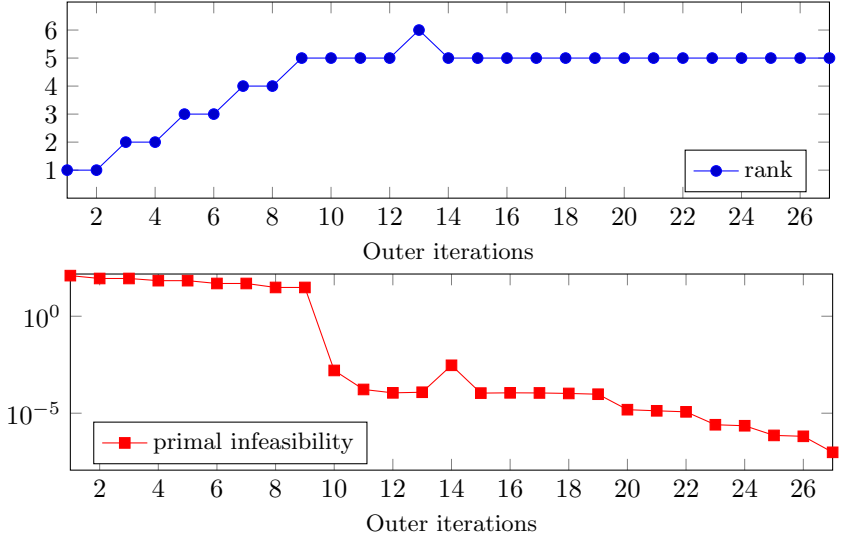
\begin{figure} \centering
\hspace*{10pt}
			\begin{tikzpicture}
		\begin{axis}[width=. 9* \textwidth, 
		   legend cell align=left,
		   legend pos=south east,
		   xlabel = {\small Outer iterations}, 
		   height = .2 \textheight,
		   ytick = {1,2,...,6},
		   xmax = 27,
		   xmin = 1,
		   ymax = 7,
		   ymin = 0
		 ]
		\foreach \j in {1} {
		  \addplot+ table[x index = 0, y index = \j] {./rank_upd.dat};
		}
		\legend{\small rank  };
		\end{axis}
	\end{tikzpicture}\\
	\begin{tikzpicture}
		\begin{semilogyaxis}[width=. 9* \textwidth, 
		   legend pos=south west,
		   xlabel = {\small Outer iterations}, 
		   height = .2 \textheight,
			 xmax = 27,
		   xmin = 1,
		   ymax = 150
		 ]
		\foreach \j in {1} {
		  \addplot+[mark options={color=red},color=red, mark=square*] table[x index = 0, y index = \j] {./primal.dat};
		}
		\legend{ \small primal infeasibility};
		\end{semilogyaxis}
	\end{tikzpicture}
  \caption{Typical behaviour of the rank update strategy described in Algorithm \ref{GF_algo_upd}.
  Data: $ \hat n=50$, target rank $r=5$, starting rank $r=1$.}
\label{fig:rankupd}
	\end{figure}
   
\vskip 10 pt 
\noindent
The proposed approach handles well the situation where the matrix 
which has to be rebuilt is nearly low-rank. We recall that by Corollary \ref{cor} we generate a low-rank 
approximation $\bar X_k$, while the primal variable $X_k$ is nearly 
low-rank and gradually approaches a low-rank solution. 
Then, at termination, we approximate the nearly low-rank matrix 
that has to be recovered with the low-rank solution approximation.   

Letting $\sigma_1\ge \sigma_2\ge \dots \ge \sigma_{\hat n}$ be 
the singular values of $B$, we perturbed each singular value of $B$
by a random scalar  $\xi = 10^{-3}\eta$, where $\eta$ is drawn 
from the standard normal distribution, and using the SVD decomposition 
of $B$ we obtain a nearly low-rank matrix $\hat B$. 
%
%
We applied {\sc IPLR-GS\_P} to \eqref{matcompl-sdp} with the aim 
to recover the nearly low-rank matrix $\hat B$ with tolerance 
in the stopping criterion set to $\epsilon =10^{-4}$.
Results reported in Table \ref{tab::2} are obtained starting from $r=1$
in the rank updating strategy. In the table we also report the rank $\bar r$ 
of the rebuilt matrix $\bar X$. The run corresponding to rank $8$, 
in italic in the table, has been performed allowing a maximum of $10$ 
inner Gauss-Seidel iterations. We observe that the method always 
rebuilt the matrix with accuracy consistent with the stopping tolerance.  
The primal infeasibility is larger than the stopping tolerance, 
as data $b$ are obtained sampling a matrix which is not low-rank 
and therefore the method does not manage to push primal infeasibility 
below $10^{-3}$. Finally we note that in some runs (rank equal to 4,5,6)
the returned matrix $\bar X$ has a rank $\bar r$ larger than that 
of the original matrix $B$.  However, in this situation we can observe 
that $\bar X$ is nearly-low rank as $\sigma_i=O(10^{-3})$, $i=r+1,\ldots,\bar r$ 
while $\sigma_i \gg 10^{-3}$, $i=1,\ldots,r$. Therefore the matrices 
are well rebuilt for each considered rank $r$ and the presence of small 
singular values does not affect the updating/downdating procedure. 

\begin{table}[htb!]
\begin{center}
\begin{tabular}{lcccccc}
\toprule
           &             \multicolumn{ 5}{c}{\sc IPLR-GS\_P }             &            \\
\midrule
  rank/$n$/$m$   & $\|{\cal A}(X)-b\|$ & $\|XS-\mu I\|_F$ & $\lambda_{\min}(S)$ &  $ \| \bar X - \hat B\|_F /\|\hat B\|_F$ &   $\hat r$ &     cpu  \\
\midrule
  3/1200/35910 &      4E-03 &      1E-03 &     4E-08 &          2E-05 &    3 &    218\\

4/1200/47840 &      { 5E-03} &   {1E-03} &  {  4E-08} &   {   2E-05} &  {\bf 5} &   {  506} \\

5/1200/59750 &       { 5E-03} &    {2E-03}   &     { 1E-07}  &      {    2E-05}     &     {\bf 7}  &  {   937}\\

6/1200/71640 &         {6E-03} &    {  1E-03}   &     {  4E-08}  &      {      2E-05}     &     {\bf 7}  &  {  797}\\

7/1200/83510 &      6E-03 &      1E-03 &     4E-08 &          2E-05 &    7 &    642 \\

8/1200/95360 &      {\it 7E-03} &     {\it  1E-03} &     {\it 4E-08} &        {\it  2E-05} &    8  &    1173 \\

       \bottomrule
\end{tabular}  
\caption{{\sc IPLR-GS\_P} starting from $r=1$ on nearly low-rank matrices ($\xi = 10^{-3}$). \label{tab::2}}
\end{center}
\end{table}

\subsection{Tests on real data sets}
In this section we discuss matrix completion problems arising 
in diverse applications as the matrix to be recovered represents 
city-to-city distances, a grayscale image, game parameters 
in a basketball tournament and total number of COVID-19 infections.

\subsection*{Low-rank approximation of partially known matrices}
We now consider an application of matrix completion where one wants 
to find a low-rank approximation of a matrix that is only partially known.

As the first test example, we consider a $312 \times 312 $ matrix taken 
from the ``City Distance Dataset'' \cite{city} and used 
in \cite{CaiCandesShen2010}, that represents the city-to-city 
distances between 312 cities in the US and Canada  computed 
from latitude/longitude data.

We sampled the 30\% of the matrix $G$ of geodesic distances and computed 
a low-rank approximation $ \bar X$ by {\sc IPLR-GS\_P} inhibiting rank 
updating/downdating and using $\epsilon=10^{-4}$. 
We compared the obtained {\color{black} solution with
the approximation $\bar X_{os}$ computed by {\sc OptSpace}  and}
the best rank-$r$ approximation 
$\bar X_r$, computed by truncated SVD ({\sc TSVD}), that requires the knowledge 
of the full matrix $G$. We considered some small values of the rank 
($r=3,4,5$) and in Table \ref{tab:real} reported the errors 
${\cal E}_{ip}=\|G- \bar X\|_F/\|G\|_F$, {\color{black}${\cal E}_{os}=\|G-\bar X_{os}\|_F/\|G\|_F$}
and ${\cal E}_r=\|G-\bar X_r\|_F/\|G\|_F$.
We  remark that the matrix $G$ is not nearly-low-rank, and {\color{black} our} method 
correctly detects that there does not exist a feasible rank $r$ matrix 
as it is not able to decrease the primal infeasibility below $1e0$.
On the other hand the error {\color{black} ${\cal E}_{ip}$} 
in the provided approximation, obtained using 
only the 23\% of the entries, is the same as that of the best rank-$r$ 
approximation $\bar X_r$. 
Note that computing the 5-rank approximation is more demanding. 
In fact the method requires on average: 3.4 Gauss-Seidel iterations, 
37 unpreconditioned CG iterations for computing $\Delta U$ 
and 18 preconditioned CG iterations for computing $\Delta y$.
In contrast, the 3-rank approximation requires on average: 
3.8 Gauss-Seidel iterations, 18 unpreconditioned CG iterations for computing 
$\Delta U$ and 10 preconditioned CG iterations for computing $\Delta y$.
As a final comment, we observe that {\sc IPLR-GS} fails when $r=5$ 
since unpreconditioned CG struggles with the solution of \eqref{sis2}. 
The computed direction $\Delta y$ is not accurate enough and the method 
fails to maintain $S$ positive definite within the maximum number 
of allowed backtracks. Applying the preconditioner cures the problem 
because more accurate directions become available.
{\color{black} Values of the error $ {\cal E}_{op}$ obtained 
with {\sc OptSpace} are larger than ${\cal E}_r$. 
However it is possible to attain comparable values for $r=3$ and $r=5$ 
under the condition that the default maximum number of iterations 
of {\sc OptSpace} is increased 10 times. In these cases, {\sc OptSpace} 
is twice and seven time faster, respectively.} 
%


\begin{table}[htb!]
\begin{center}
\begin{tabular}{ccc|ccccc}

\toprule
           & {{\sc TSVD} }         & {{\sc OptSpace} }     & \multicolumn{ 4}{c}{{\sc IPLR-GS\_P} }                     \\
\midrule
         rank & ${\cal E}_r$ & ${\cal E}_{op}$ & ${\cal E}_{ip}$ & $\|{\cal A}(X)-b\|$ & $\|XS-\mu I\|_F$ & $\lambda_{\min}(S)$  &       cpu  \\
\midrule
         3 &  1.15E-01  &1.97E-01 & 1.23E-01 &   4E00 &  8E-04 &  4E-07 &    48  \\

         4 &  7.06E-02  & 1.99E-01 &7.85E-02 &   3E00 &  8E-04 &  4E-07 &    70 \\

         5 &  5.45E-02  & 1.30E-01 &6.01E-02 &   2E00 &  8E-04 &  4E-07 &   243 \\ 

\bottomrule

\end{tabular}  
\caption{{\sc TSVD}, {\color{black}{\sc OptSpace} and} {\sc IPLR-GS\_P} 
for low rank approximation of the City Distance matrix. \label{tab:real}}
\end{center}
\end{table}

As the second test example, we consider the problem of computing a low rank 
approximation of an image that is only partially known because some pixels 
are missing and we analyzed the cases when the missing pixels are 
distributed both randomly and not randomly (inpainting).
To this purpose, we examined the {\em Lake} $512\times 512$ original 
grayscale image
   \footnote{The Lake image can be downloaded 
   from \url{http://www.imageprocessingplace.com}.} 
shown in Figure \ref{fig:lakea} and generated the inpainted versions 
with the 50\% of random missing pixels (Figure \ref{fig:lakeb})
and with the predetermined missing pixels (Figure \ref{fig:lakec}).
\begin{figure}[h] \centering
 \subfloat[True image]{
	\includegraphics[width=0.31\textwidth]{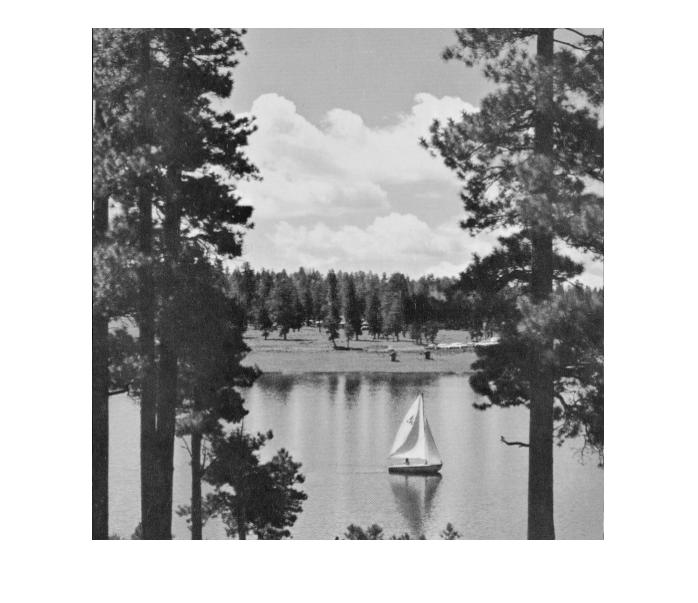}   \label{fig:lakea}
	}
	 \subfloat[50\% random missing pixels]{
	\includegraphics[width=0.31\textwidth]{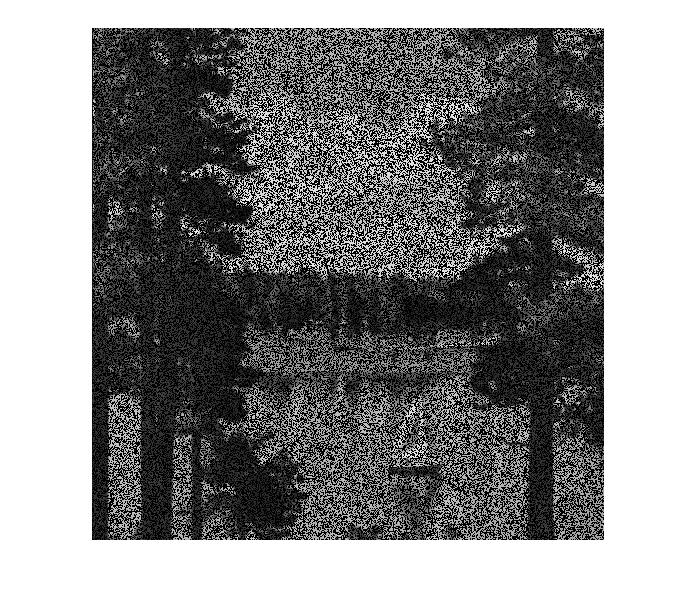}  \label{fig:lakeb}
 	}
		 \subfloat[7\% nonrandom missing pixels]{
	\includegraphics[width=0.31\textwidth]{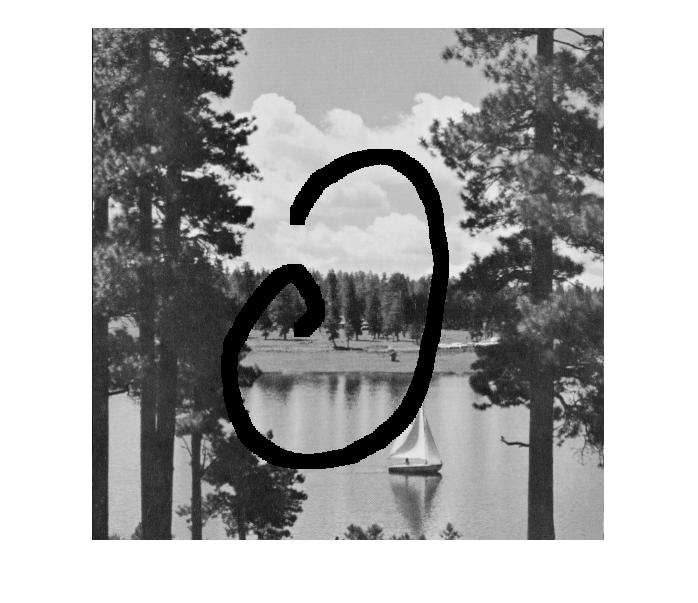}  \label{fig:lakec}
	}
  \caption{The Lake test true image and the inpainted versions.} 
\end{figure}

We performed tests fixing the rank to values ranging from 10 to 150
and therefore used {\sc IPLR-BB} which is computationally less sensitive 
than {\sc IPLR-GS} to the magnitude of the rank.

In Figure \ref{fig:psnr_lake} we plot the quality of the reconstruction 
in terms of relative error $ {\cal E}$ and PSNR (Peak-Signal-to-Noise-Ratio)
against the rank, for {\sc IPLR-BB}, {\color{black}{\sc OptSpace}} and truncated SVD.
We observe that when the rank is lower than 40, 
 {\color{black}{\sc IPLR-BB} and {\sc TSVD}} 
give comparable results, but when the rank increases the quality
obtained with {\sc IPLR-BB} does not improve. As expected, 
by adding error information available only from the knowledge
of the full matrix, the truncated SVD continues to improve 
the accuracy as the rank increases. 
{\color{black} The reconstructions produced with {\sc OptSpace} 
display noticeably worse values of the two relative errors
(that is, larger ${\cal E}$ and smaller PSNR, respectively) 
despite the rank increase.}
	
Figure \ref{fig:inp_lake} shows that {\sc IPLR-BB} is able to recover 
the inpainted image in Figure \ref{fig:lakec} and that visually 
the quality of the reconstruction benefits from a larger rank.
{\color{black} Images restored by {\sc OptSpace} are not reported 
since the relative PSNR values are approximately 10 points lower 
than those obtained with {\sc IPLR-BB}. The quality of the reconstruction 
of images \ref{fig:lakeb} and \ref{fig:lakec} obtained with {\sc OptSpace} 
cannot be improved even if the maximum number of iterations is increased 
tenfold.} 


\begin{figure} \centering
		\begin{tikzpicture}
		\begin{axis}[width=. 5* \textwidth, 
		   title = {$ {\cal E}$}, 
		   legend pos=south west,
		   legend cell align=left,
		   xlabel = {\small Rank}, 
		   height = .25 \textheight,
		   xmax = 100,
		   xmin = 5,
		 ]
		  \addplot+[mark=triangle*,brown,thick,mark options={ fill=brown}] table[x index = 0, y index = 1] {./rank_Error_50perc_lake_R1_2.dat};
		   \addplot+ [mark=*,blue,thick,mark options={ fill=blue}]table[x index = 0, y index = 2] {./rank_Error_50perc_lake_R1_2.dat};
		    \addplot+ [mark=square*,red,thick,mark options={ fill=red}]table[x index = 0, y index = 3] {./rank_Error_50perc_lake_R1_2.dat};
	\legend{\footnotesize {\sc TSVD}, \footnotesize  {\sc IPLR-BB},\footnotesize {\sc OptSpace}};\end{axis}
		
	\end{tikzpicture}		
			\begin{tikzpicture}
		\begin{axis}[width=. 5* \textwidth, 
		   title = {PSNR}, 
		   legend pos=north west,
		   legend cell align=left,
		   xlabel = {\small Rank}, 
		   height = .25 \textheight,
		   xmax = 100,
		   xmin = 5,
		 ]
		  \addplot+ [mark=triangle*,brown,thick,mark options={ fill=brown}]table[x index = 0, y index = 1] {./rank_psnr_50perc_lake_R1_2.dat};
		  \addplot+[mark=*,blue,thick,mark options={ fill=blue}]  table[x index = 0, y index = 2] {./rank_psnr_50perc_lake_R1_2.dat};
		  \addplot+[mark=square*,red,thick,mark options={ fill=red}] table[x index = 0, y index = 3] {./rank_psnr_50perc_lake_R1_2.dat};
\legend{\footnotesize {\sc TSVD}, \footnotesize  {\sc IPLR-BB},\footnotesize {\sc OptSpace}}\end{axis}
		
	\end{tikzpicture}		
  \caption{Rank versus error and PSNR of the Lake image recovered with
  truncated SVD ({\sc TSVD}), {\sc IPLR-BB} and {\color{black} \sc OptSpace} (50\% random missing pixels in Figure \ref{fig:lakeb}). }
\label{fig:psnr_lake}
	\end{figure}
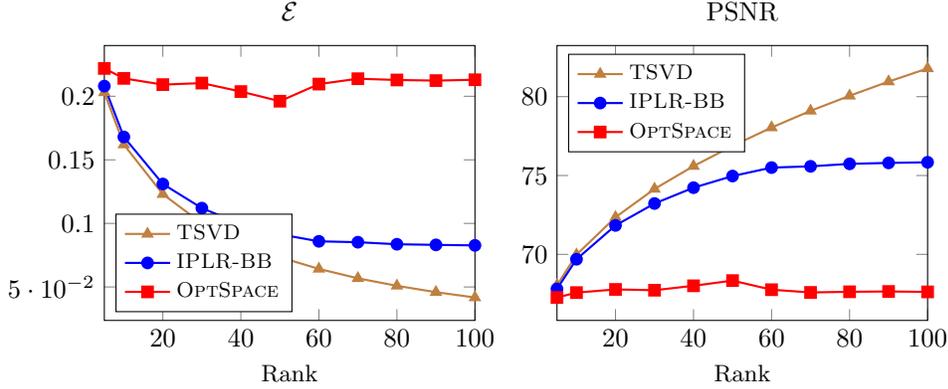

\begin{figure} 
\captionsetup[subfigure]{justification=centering}\centering
		 \subfloat[
	 $r = 80$ 
		 PSNR = 76.15, $ {\cal E}$= 1.16E-01]
		 [{\sc IPLR-BB}: $r = 80$ \\
		 PSNR = 76.15, $ {\cal E}$= 1.16E-01]{
 	\includegraphics[width=0.32\textwidth]{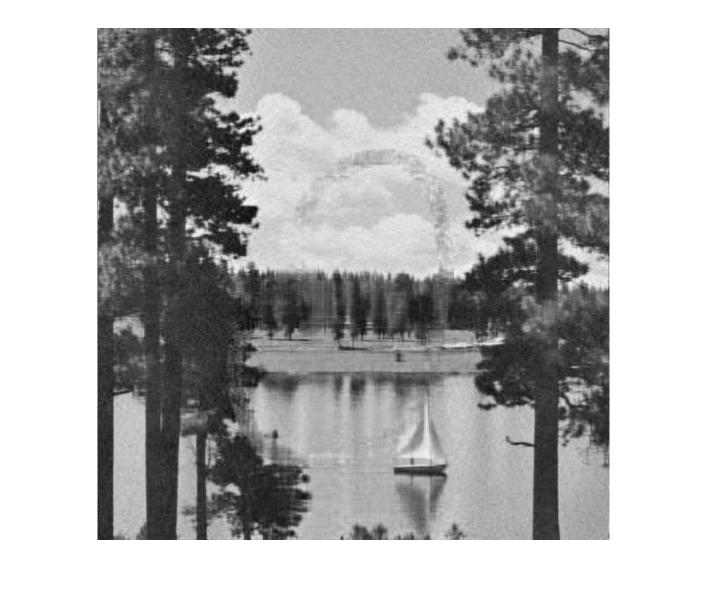}  
 	}
		 \subfloat[
		 $r = 100$ 
		 PSNR = 76.64, $ {\cal E}$= 7.53E-01]
		 [{\sc IPLR-BB}: $r = 100$  \\
		 PSNR = 76.64, $ {\cal E}$= 7.53E-01]{
	\includegraphics[width=0.32\textwidth]{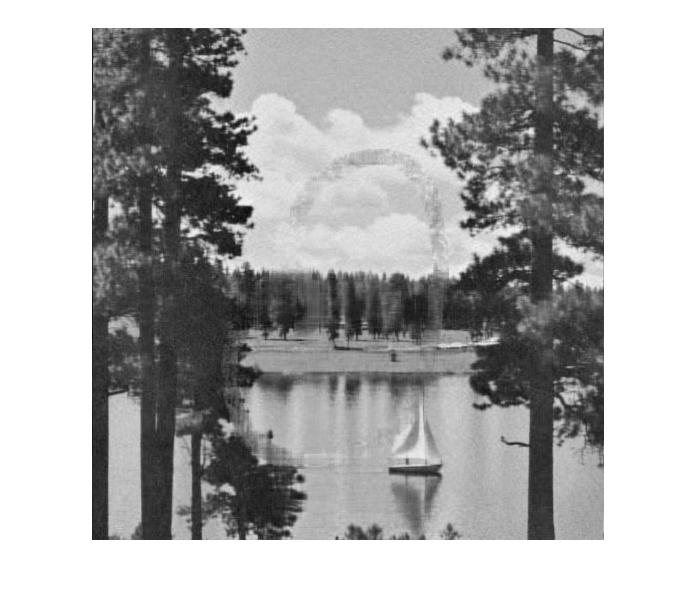}  
	}
		 \subfloat[$r = 150$ 
		 PSNR = 78.44, $ {\cal E}$= 6.12E-02]
		 [{\sc IPLR-BB}: $r = 150$ \\
		 PSNR = 78.44, $ {\cal E}$= 6.12E-02]{
	\includegraphics[width=0.32\textwidth]{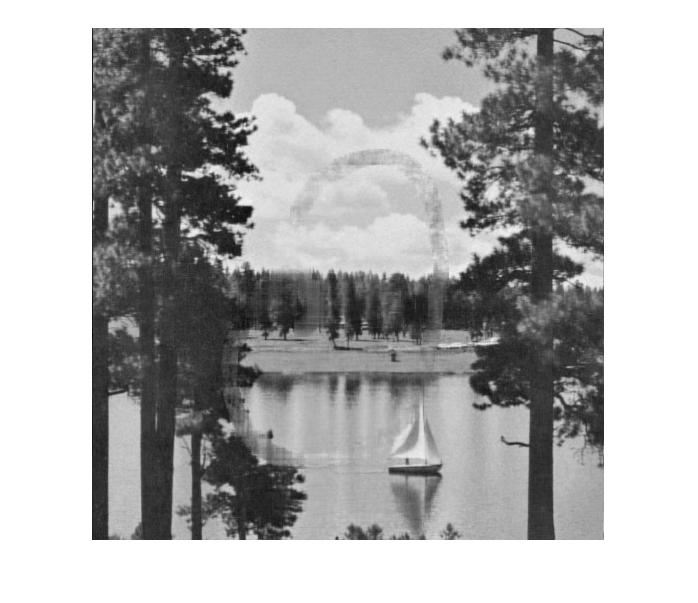}  
	} 
%
  \caption{Images recovered by {\sc IPLR-BB} for different rank values
  and corresponding PSNR and error (non-random missing pixels in Figure  \ref{fig:lakec}).}
  \label{fig:inp_lake}
\end{figure}

\subsection*{Application to sports game results predictions}	
Matrix completion is used in sport predictive models to forecast 
match statistics \cite{ji2015march}. We consider the dataset 
concerning the NCAA Men's Division I Basketball Championship, 
in which each year 364 teams participate.\footnote{The March Machine Learning Mania 
    dataset is available in the website 
    \url{https://www.kaggle.com/c/march-machine-learning-mania-2016/data}.}. 
The championship is organized in 32 groups, called Conferences, 
whose winning teams face each other in a final single elimination 
tournament, called March Madness. Knowing match statistics of games 
played in the regular Championship, the aim is to forecast the potential 
statistics of the missing matches played in the March Madness phase. 
In our tests, we have selected one match statistic of the 2015 Championship, 
namely the fields goals attempted (FGA) and have built a matrix where 
teams are placed on rows and columns and nonzero $ij$-values correspond
to the FGA made by team $i$ and against team $j$.
In this season, only 3771 matches were held and therefore we obtained 
a rather sparse $364\times 364$ matrix of FGA statistics; in fact, only 
the 5.7\% of entries of the matrix that has to be predicted is known. 
To validate the quality of our predictions we used the statistics  
of the 134 matches actually played by the teams in March Madness. 
We verified that in order to obtain reasonable predictions 
of the missing statistics the rank of the recovered matrix has 
to be sufficiently large. 
Therefore we use {\sc IPLR-BB} setting the starting rank $r=20$, rank increment $\delta_r=10$ 
and $\epsilon=10^{-3}$. The algorithm  terminated  recovering matrix $\bar X$ of rank 30.
%
%
In Figure \ref{fig:FGA} we report the bar plot of the exact and predicted 
values for each March Madness match. The matches  have been numbered 
from 1 to 134.  We note that except for 12 mispredicted statistics, 
the number of fields goals attempted is predicted reasonably well. 
In fact, we notice that the relative error between the true and 
the predicted statistic is smaller than $20\%$ in the 90\% 
of predictions. 

{\color{black} On this data set, {\sc OptSpace} gave similar results 
to those in Figure \ref{fig:FGA} returning a matrix of rank 2.}

\begin{figure}
	\centering
		\includegraphics[width=0.48\textwidth]{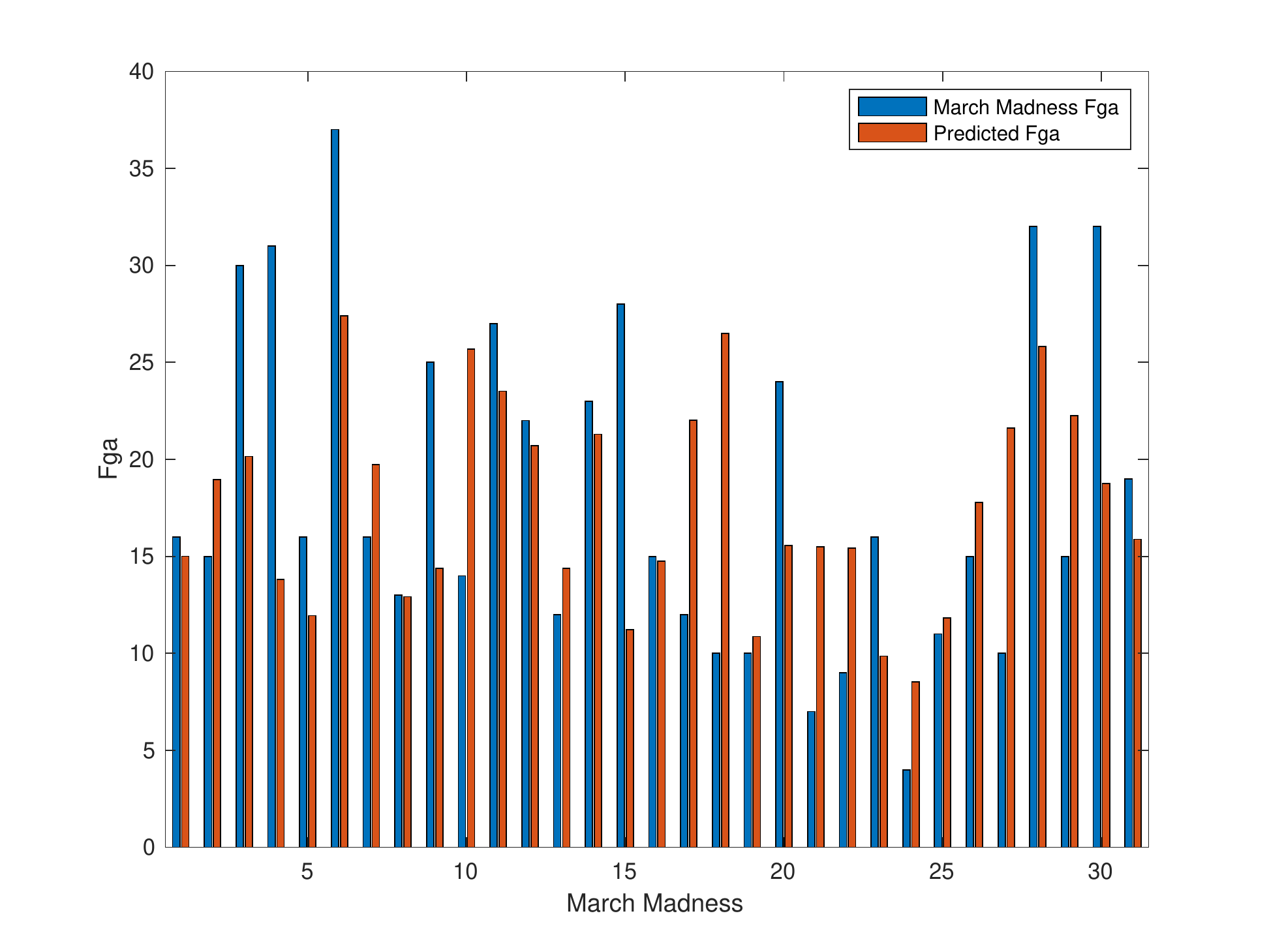}
		\includegraphics[width=0.48\textwidth]{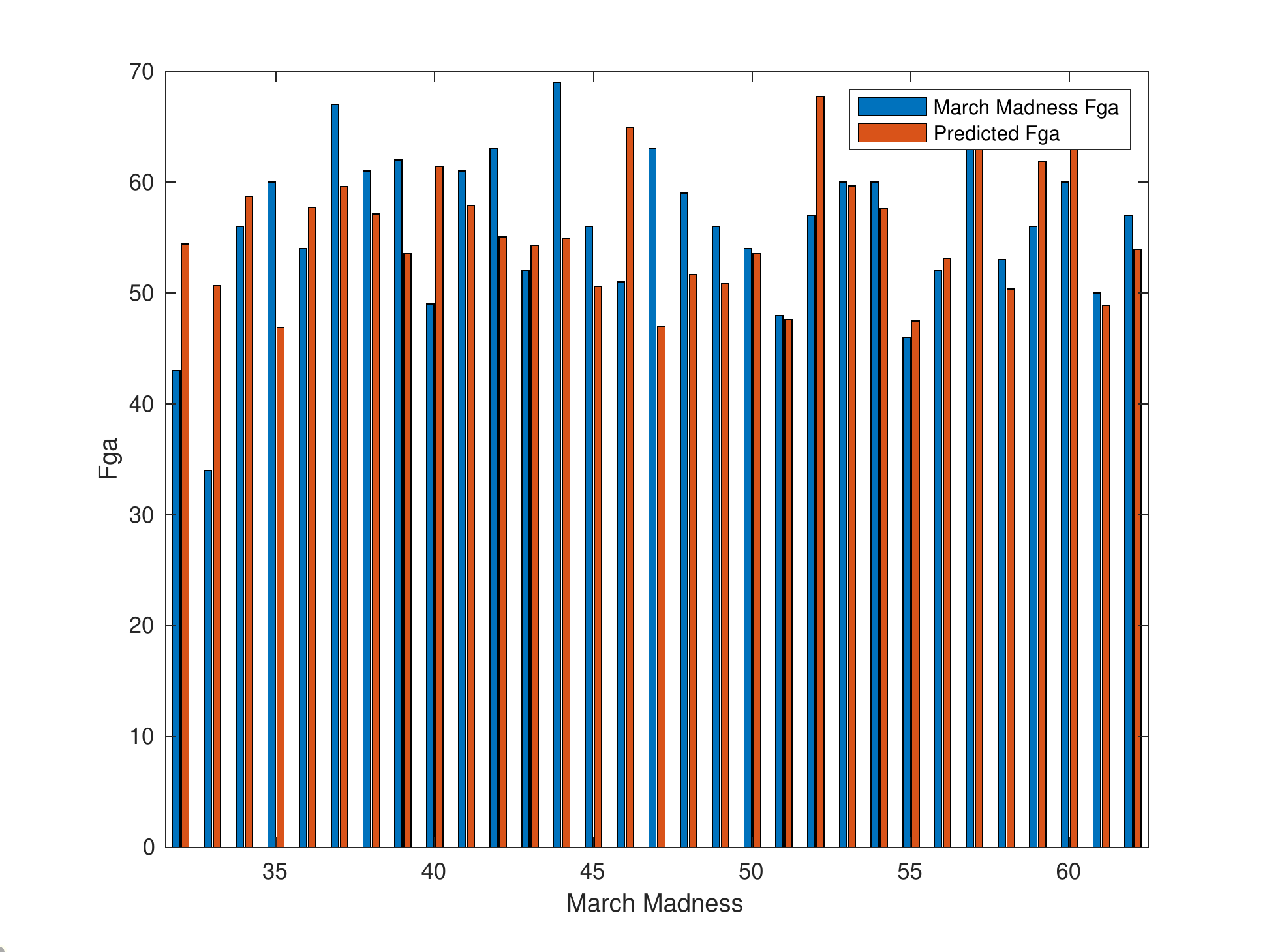}
		\includegraphics[width=0.48\textwidth]{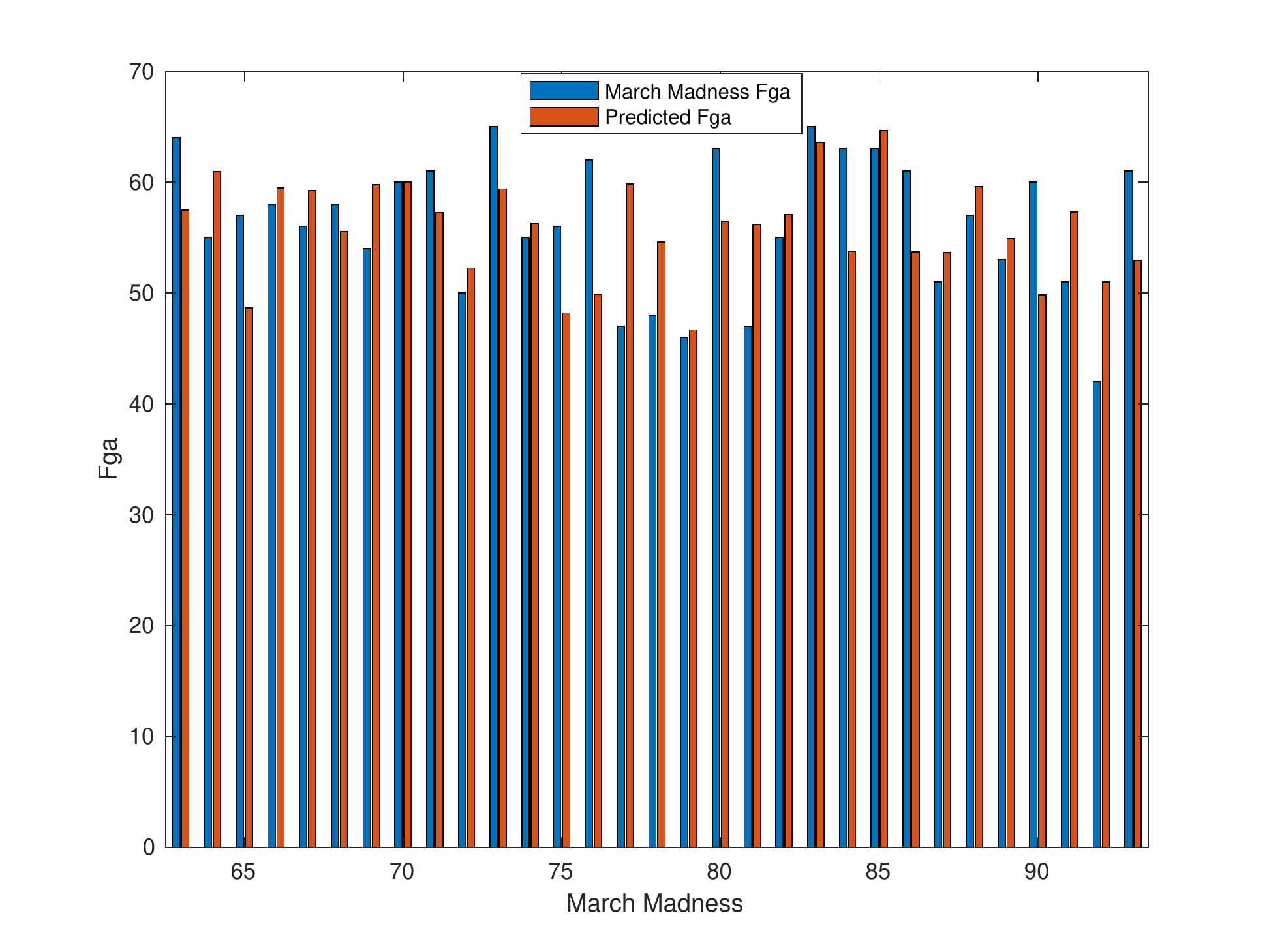}
	\includegraphics[width=0.48\textwidth]{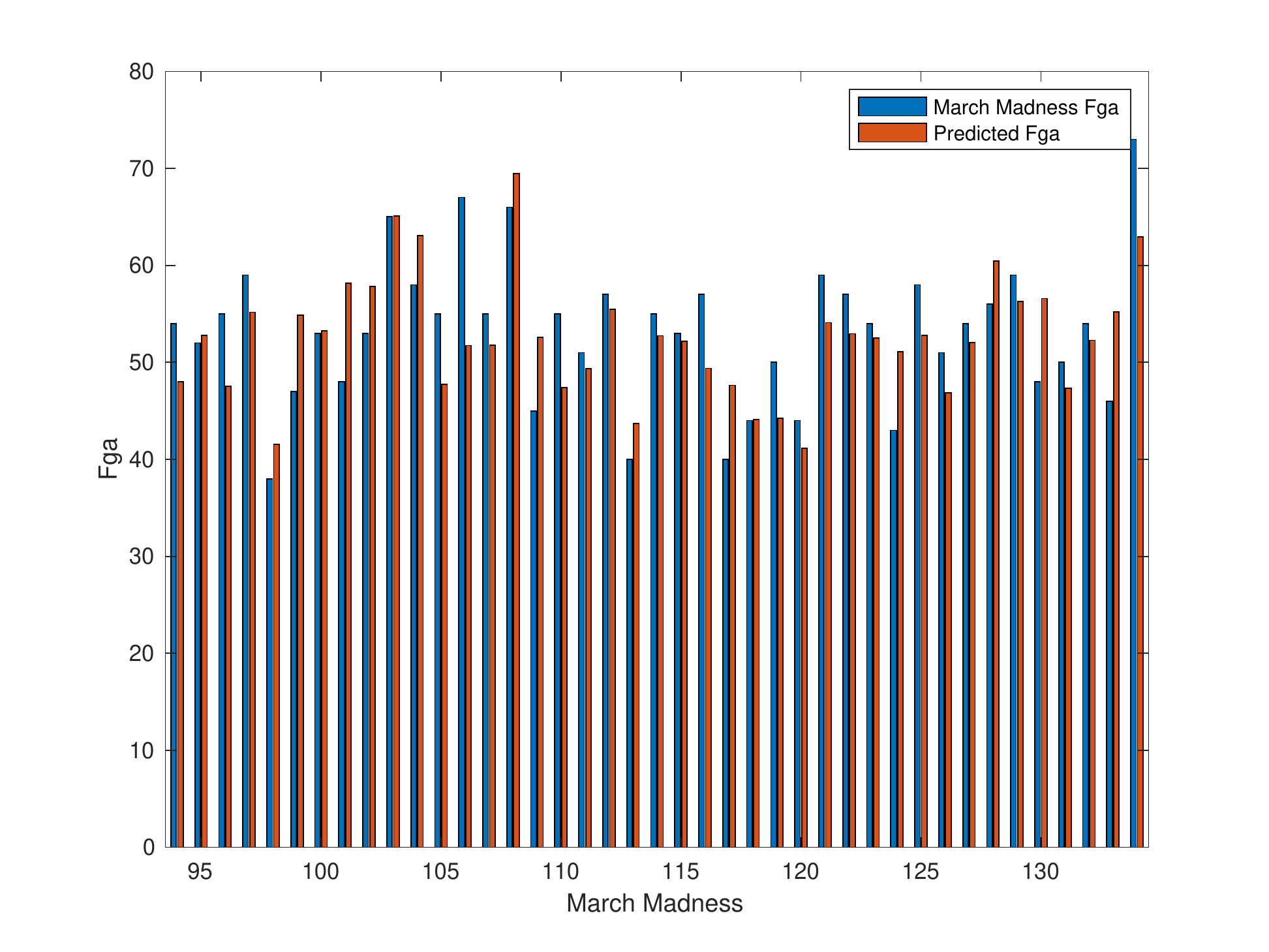}

	\caption{Predicted and March Madness FGA statistics. 
        Top-Left: matches 1 to 31, 
        Top-Right matches 32 to 62, 
        Bottom-Left matches 63 to 93,
	Bottom-Right matches 94 to 134.}
    \label{fig:FGA}
\end{figure}
 
\subsection*{Application to COVID-19 infections missing data recovery}	

We now describe a matrix completion problem where data are the number 
of COVID-19 infections in provincial capitals of regions in the North of Italy.
Each row and column of the matrix corresponds to a city and to a day, 
respectively, so that the $ij$-value corresponds to the total number 
of infected people in the city $i$ on the day $j$.
We used data made available by the Italian Protezione Civile
   \footnote{The dataset is available at 
   \url{https://github.com/pcm-dpc/COVID-19/tree/master/dati-province}.} 
regarding the period between March 11th and April 4th 2020, that is, 
after restrictive measures have been imposed by the Italian Government 
until the current date.
We assume that a small percentage (5\%) of data is not available 
to simulate the real case because occasionally certain laboratories 
do not communicate data to the central board.
In such a case our aim is to recover this missing data and provide 
an estimate of the complete set of data to be used to make 
analysis and forecasts of the COVID-19 spread.
Overall, we build a $47 \times 24$ dense matrix and attempt 
to recover  56 missing entries in it.
{\color{black}
We use {\sc IPLR-GS\_P} with starting rank $r=2$, rank increment $\delta_r=1$ 
and $\epsilon=10^{-4}$ and we have obtained a matrix $\bar X$ of rank 2.
The same rank is obtained using {\sc OptSpace} but only if the maximum 
number of its iterations is increased threefold. 
In Figure \ref{fig:covid} both the predicted and actual data (top) 
and the percentage error (bottom) are plotted using the two solvers.
We observe that {\sc IPLR-GS\_P} yields an error below 10\% except 
for 8 cases and in the worst case it reaches 22\%. 
The error obtained with {\sc OptSpace} exceeds 10\% in 15 cases 
and in one case reaches 37\%.}

The good results obtained with  {\sc IPLR-GS\_P} for this small example are encouraging 
for applying the matrix completion approach to larger scale data sets.  


\begin{figure}
   \centering
   \includegraphics[width=0.98\textwidth]{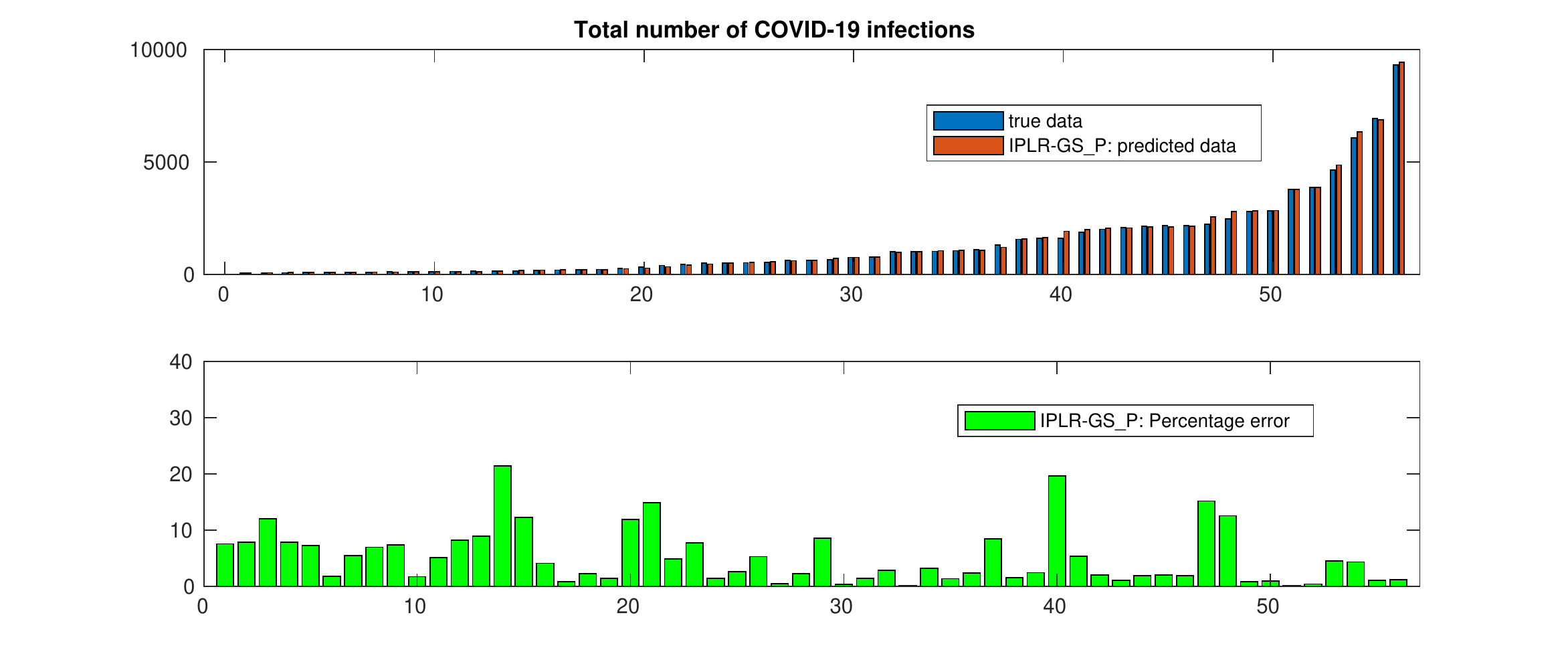} 
   \includegraphics[width=0.98\textwidth]{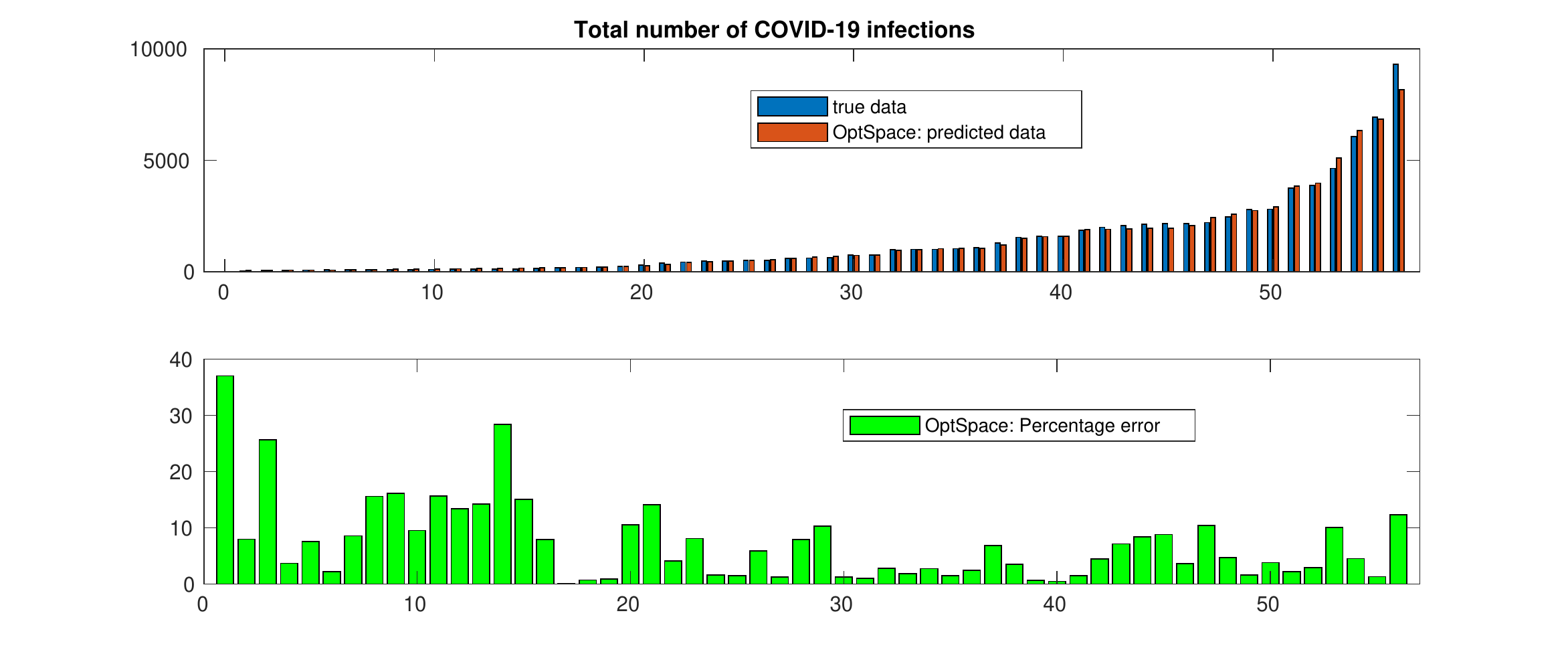} 
   \caption{Predicted and actual number of COVID-19 infections (top) 
            and corresponding percentage error{\color{black}, 
            obtained with {\sc IPLR-GS\_P} (2 top plots) 
            and {\sc OptSpace} (2 bottom plots).}}
   \label{fig:covid}
\end{figure}

\section{Conclusions} 
\label{Concls} 

We have presented a new framework for an interior point method 
for low-rank semidefinite programming. The method relaxes the rigid 
IPM structure and replaces the general matrix $X$ with the special 
form \eqref{Xlow} which by construction enforces a convergence 
to a low rank solution as $\mu$ goes to zero. Therefore effectively 
instead of requiring a general $n \times n$ object, the proposed 
method works with an $n \times r$ matrix $U$, which delivers 
significant storage and cpu time savings. It also handles well problems 
with noisy data and allows to adaptively correct the (unknown) rank.
We performed extensive numerical results on SDP reformulation 
of matrix completion problems using both the first- 
and the second-order methods to compute search directions. 
The convergence of the method has been analysed under the assumption that 
eventually the steplength $\alpha_k$ is equal to one (Assumption 1).
However, this seemingly strong assumption does hold in all our numerical 
tests except for the sports game results predictions where the number 
of known entries of the matrix is extremely low.

%
%
Our numerical experience shows the efficiency of the proposed method 
and its ability to handle large scale matrix completion problems 
and medium scale problems arising in real-life applications.
{\color{black}
A comparison with {\sc OptSpace} reveals that the proposed method 
is versatile and it delivers more accurate solutions when applied 
to ill-conditioned or to some classes of real-life applications. 
It is generally slower than  methods specially designed for matrix 
completion as {\sc OptSpace}, but our method has potentially a wider 
applicability.}
%

\appendix
\section{Notes on Kronecker product and matrix calculus}\label{sec:app}
Let us also recall several useful formulae which involve Kronecker products. 
For each of them, we assume that matrix dimensions are consistent 
with the multiplications involved. 

Let $A,B,C,D$ be matrices of suitable dimensions. Then
\begin{eqnarray}
(A \otimes B) (C\otimes D) & = & (AC \otimes BD)  \label{kron1} \\
 vec(AXB) & = & (B^T \otimes A) vec(X)  \label{kron2} \\
 vec(AX^TB) & = & (B^T \otimes A) vec(X^T) = (B^T \otimes A) \Pi vec(X) \label{kron3},
\end{eqnarray}
where $\Pi$ is a permutation matrix which transforms $vec(X)$ to $vec(X^T)$.
Moreover, assume that $A$ and $B$ are square matrices of size $n$ and $m$ respectively. 
Let $\lambda_1, \dots, \lambda_n$ be the eigenvalues of $A$ and 
$\mu_1, \dots, \mu_m$ be those of $B$ (listed according to multiplicity). 
Then the eigenvalues of $A \otimes B $ are
$$
\lambda _{i}\mu _{j},\qquad i=1,\ldots ,n,\,j=1,\ldots ,m.
$$

Finally, following \cite{deriv}, we recall some rules for derivatives 
of matrices that can be easily derived applying the standard derivation 
rules for vector functions (chain rule, composite functions) and 
identifying $d\, \mathcal{G}(X)/d\,(X)$ by using the vectorization
$d \,vec \mathcal{G}(X)/ d\, vec(X)$, where $\mathcal{G}(X)$ 
is a matrix function. In particular we have that given the matrices 
$A \in \mathbb{R}^{n\times m}$, $B \in \mathbb{R}^{p\times q}$
and $X$ defined accordingly, it holds
\begin{eqnarray*}
  \frac{d\, AA^T}{d\, A} & = & (A \otimes I_n) + (I_n\otimes A),\\
  \frac{d\, AXB}{d\,X} & = & (B^T \otimes A).
\end{eqnarray*}

\bibliographystyle{siam}
\bibliography{mybiblio_sdp}

\end{document}